\documentclass[11pt,reqno]{amsart}
\usepackage[utf8x]{inputenc}

\usepackage{hyperref}
\hypersetup{
    colorlinks=true,       % false: boxed links; true: colored links
    linkcolor=blue,          % color of internal links
    citecolor=blue,        % color of links to bibliography
    filecolor=blue,      % color of file links
    urlcolor=cyan  
}
\usepackage{latexsym,amssymb,amsmath,amsthm,bbm}
\usepackage{epsfig,mathrsfs}

\def\RR{{\mathbb R}}
\newcommand{\R}{{\ensuremath{\mathbb{R}}}}
\newcommand{\N}{{\ensuremath{\mathbb{N}}}}

\renewcommand{\P}{\ensuremath{\mathbb{P}}}
\renewcommand{\dj}{d\kern-0.4em\char"16\kern-0.1em}
\newcommand{\E}{\ensuremath{\mathbb{E}}}

\newtheorem{thm}{Theorem}[section]
\newtheorem{lemma}[thm]{Lemma}
\newtheorem{defn}[thm]{Definition}
\newtheorem{prop}[thm]{Proposition}
\newtheorem{corollary}[thm]{Corollary}
\newtheorem{remark}[thm]{Remark}

\numberwithin{equation}{section}

\newcommand{\cal}[1]{\mathcal{#1}}

\def\sA {{\cal A}}

  \def\sL {{\cal L}}

  \def\bR {{\mathbb R}}

\def\eps{\varepsilon}
\def\BB{{\cal B}}
\def\wh{\widehat}
\def\wt{\widetilde}
\def\pf{\noindent{\bf Proof.} }
\def\beq{\begin{equation}}
\def\eeq{\end{equation}}
\def\bee{\begin{equation}}
\def\eee{\end{equation}}
\def\nn{\nonumber}

\begin{document}
\numberwithin{equation}{section}
\bibliographystyle{amsalpha}

\title[Green function estimates for SBM]{Green function estimates for subordinate Brownian motions : stable and beyond}
\begin{abstract}
A subordinate Brownian motion $X$ is a L\'evy process which can be
obtained by replacing the time of the Brownian motion by an
 independent subordinator.
In this paper, 
when the Laplace exponent $\phi$
of the corresponding subordinator satisfies some mild conditions, we first prove the scale
invariant boundary Harnack inequality for $X$ on arbitrary open sets. Then we give an explicit form of sharp two-sided estimates 
of  the
Green functions of these subordinate Brownian motions in any bounded
 $C^{1,1}$ open set.
As a consequence, we prove the boundary Harnack inequality for $X$ on any
 $C^{1,1}$ open set  with explicit decay rate.
Unlike \cite{KSV2, KSV4}, our results cover geometric stable processes and relativistic geometric stable process, i.e. the cases when the subordinator has the Laplace exponent
\[
\phi(\lambda)=\log(1+\lambda^{\alpha/2})\ \ \ \ 
(0<\alpha\leq 2, d > \alpha)\]
and
\[
\phi(\lambda)=\log(1+(\lambda+m^{\alpha/2})^{2/\alpha}-m)\ \ \ \ (0<\alpha<2,\,
m>0, 
d >2)\,.
\]
\end{abstract}

\author{Panki Kim}
\address{Department of Mathematical Sciences and Research Institute of Mathematics,
Seoul National University,
Building 27, 1 Gwanak-ro, Gwanak-gu
Seoul 151-747, Republic of Korea}
\curraddr{}
\email{pkim@snu.ac.kr}
\thanks{Research supported by Basic Science Research Program through
the National Research Foundation of Korea(NRF) funded by the Ministry of Education,
Science and Technology(0409-20120034).}

\author{Ante Mimica}
\address{Fakult\"{a}t f\"{u}r Mathematik, Universit\"{a}t Bielefeld, Postfach
100131, D-33501 Bielefeld, Germany}
\curraddr{}
\thanks{ Research supported in part by German Science Foundation DFG via IGK
"Stochastics and real world models" and SFB 701.}
\email{amimica@math.uni-bielefeld.de}
\thanks{}



\maketitle

\section{Introduction}
Let $d$ be a positive integer, let $W=(W_t,\P_x)$ be  a Brownian motion in
$\R^d$ starting at $x$ 
  and let  $S=(S_t\colon t\geq 0)$ 
be a subordinator independent of $W$,
i.e. a L\' evy process taking values in $[0,\infty)$ and starting at $0$. 

 The Laplace exponent of a subordinator is a Bernstein
function and hence has the representation
\begin{equation}\label{eq:sub}
\phi(\lambda)=b\lambda +\int\limits_{(0,\infty)}(1-e^{-\lambda t})\, \mu(dt)\, ,
\end{equation}
where $b\ge 0$ and $\mu$ is a measure on $(0,\infty)$
satisfying
$\int\limits_{(0,\infty)}(1\wedge t)\, \mu(dt)<\infty$, usually called the L\'
evy measure of $\phi$.
If the measure $\mu$ has a completely monotone density, the Laplace
exponent $\phi$ is called a complete Bernstein function.

We define the subordinate Brownian motion $X=(X_t,\P_x)$ 
by $X_t=W_{S_t}$. 

The aim of this paper is to obtain 
the following two-sided
 estimates of the Green function 
  $G_D(x,y)$ of $X$ in 
a bounded $C^{1,1}$ open set $D\subset \R^d$ in terms of the Laplace exponent
$\phi$ of the subordinator:
\[
 G_D(x,y)\asymp \left(1 \wedge
\frac{\phi(|x-y|^{-2})}{\sqrt{\phi(\delta_D(x)^{-2})
\phi(\delta_D(y)^{-2})}}\right)\,
\frac{\phi'(|x-y|^{-2})}{|x-y|^{d+2}\phi(|x-y|^{-2})^2}\,,
\]
where $\delta_D(x)$ denotes the distance of the point $x$ to 
 $D^c$ and $a\wedge b:=\min \{a, b\}$. 
Here and in the sequel,  $f\asymp g$  means that the quotient 
$\frac{f}{g}$
stays bounded
between two positive numbers on their common domain of definition.

The process 
 $X$ is, in particular,   a rotationally symmetric L\' evy process. Recently
there has been huge interest in studying
potential theory 
of such processes.  See,
for example, \cite{KMR, KSV3, KSV2, KSV4, RSV} and 
references therein. The
purpose of this paper is to extend recent results in \cite{KSV2, KSV4} by
covering  geometric stable processes and much more. 

Estimates of Green function for discontinuous Markov processes 
were first
studied for rotationally symmetric $\alpha$-stable processes in \cite{CS4} and
in \cite{Kul} independently. These results were extended later to relativistic
$\alpha$-stable processes and to sums of two independent stable processes in
\cite{Ryz} and \cite{CKS2} respectively. 
 Recently, the first named author with R. Song and Z.
Vondra\v cek succeeded to obtain such estimates for a large class of subordinate
Brownian motions in \cite{KSV2}. 

Still, 
the class considered in \cite{KSV2}
does not include some
interesting cases like geometric stable processes or, more generally, the class
of subordinate Brownian motions with Laplace exponent that varies slowly at
infinity. Our approach covers a large class of such processes. 

Another feature of our approach is that it is unifying in the following sense: 
the sharp estimates of the Green function are given only in terms of the Laplace
exponent $\phi$ and its derivative. 

Let us give a few examples of transient processes that are covered by our
approach. 

{\bf Example 1} (Geometric stable processes)
\[\phi(\lambda)=\log(1+\lambda^{\beta/2}),\ \  \ (0<\beta\leq 2,\, d>\beta).\]

{\bf Example 2} (Iterated geometric stable processes)
\begin{align*}
 \phi_1(\lambda)&=\log(1+\lambda^{\beta/2})  \ \ \ (0<\beta\leq 2)\\
\phi_{n+1}&=\phi_1\circ \phi_n \ \ \  n\in \N,
\end{align*}
with an additional condition $d>2^{1-n}\beta^n$.

{\bf Example 3} (Relativistic geometric stable processes)
\[\phi(\lambda)=\log\left(1+\left(\lambda+m^{\beta/2}\right)^{2/\beta}-m
\right)  \ \ \ (m>0,\,0<\beta<2,\, d>2).\]

In order to obtain the sharp Green function estimates we first obtain the
uniform boundary Harnack principle, with constant not depending on the open set
itself. Such uniform boundary Harnack principle was first proved in \cite{BKK}
and very recently generalized to a larger class of rotationally symmetric L\'
evy processes in \cite{KSV4}. We adapt the approach in the latter paper in
order to cover the class of subordinate Brownian motions with slowly varying
Laplace exponents. Unlike the approach in \cite{KSV4}, instead of the use of the
Harnack inequality, we use estimates of the Green function of balls near
 boundary
obtained in \cite{KM}.

Further, our uniform 
boundary Harnack principle can be used to prove
sharp
Green function estimates for bounded $C^{1,1}$ open sets by adapting the method
in 
\cite{KSV2}. Even though we follow the roadmap in 
\cite{KSV2}, 
we needed to
make  significant changes due to the fact that now we do not have
necessarily regularly varying Laplace exponents. 

To overcome such difficulties we use new types of estimates (not only in
terms of the Laplace exponent itself, but also in terms of its derivative)  of
the jumping kernel and the potential kernel of the subordinate Brownian motions,
which were obtained for the first time in \cite{KM}.
 This type of estimates is essential in our
approach. 

Let us be more precise now. 
In this paper 
we will always assume the following three conditions on the Laplace exponent
$\phi$ of the subordinator $S$:

\noindent {\bf (A-1)}
	$\phi$ is a complete Bernstein function;
	
\noindent	{\bf(A-2)} the L\' evy density $\mu$ of $\phi$ is infinite, i.e.
$\mu(0,\infty)=\infty$ ;

\noindent
{\bf(A-3)}
	there exist constants $\sigma>0$, $\lambda_0>0$ and
$\delta \in (0, 1]$ such that
\begin{equation}\label{eq:as-1}
  \frac{\phi'(\lambda x)}{\phi'(\lambda)}\leq
\sigma\,x^{-\delta}\ \text{ for all
}\ x\geq 1\ \text{ and }\ \lambda\geq\lambda_0\,.
\end{equation}

\bigskip

Either in the case $d \le 2$ and $ \delta>1-d/2$ or in the case $d \ge 2$ and  $0<\delta \le \frac{1}{2}$	
we will sometimes 
further
assume 
two 
technical conditions below. Note that {\bf (A-4)}, related to transience of the corresponding subordinate Brownian motion,             is used in \cite{KM} to obtain the asymptotic of the jumping kernel and the Green function of   the  subordinate Brownian motion. Unlike \cite{KM} we state {\bf (A-4)} for $d=2$ and $d=1$ separately to make it clear.

  \noindent {\bf (A-4)}
If $d = 2$, we assume that there are $
\sigma_0>0$ and  $\delta_0 \in (0,2)$ such that 
\begin{equation}\label{e:new23}
  \frac{\phi'(\lambda x)}{\phi'(\lambda)}\geq
\sigma_0\,x^{-
\delta_0}\ \text{ for all
}\ x\geq 1\ \text{ and }\ \lambda\geq\lambda_0.
\end{equation}
If $d=1$, we assume that the constant $\delta$ in
{\bf (A-3)} satisfies $\delta>\tfrac{1}{2}$ and that 
there are $
\sigma_0>0$ and  $\delta_0  \in  (\tfrac{1}{2}, 2\delta-\tfrac{1}{2})$
 such that \eqref{e:new23} holds.

\noindent {\bf(A-5)}
  If $d \ge 2$ and the constant  $\delta$ in {\bf(A-3)} satisfies  $0<\delta \le
\frac{1}{2}$, then we assume that  there exist constants $\sigma_1>0$ and
$\delta_1 \in [\delta, 1 )$ 
such that
\begin{equation}\label{eq:as-2}
  \frac{\phi(\lambda x)}{\phi(\lambda)}\geq
\sigma_1\,x^{1-\delta_1}\ \text{ for all
}\ x\geq 1\ \text{ and }\ \lambda\geq\lambda_0\,.
\end{equation}

\bigskip
\begin{remark}
\begin{itemize}
\item[(a)]
Note that {\bf (A-3)} implies $b=0$ in \eqref{eq:sub}, by letting $\lambda\to \infty$.
\item[(b)]
	The condition {\bf (A-3)} is implied by the following stronger condition
  \begin{equation}\label{eq:stri-cond}
    \forall\, x>0\ \ \ \ \lim_{\lambda\to\infty}\frac{\phi'(\lambda
x)}{\phi'(\lambda)}=x^{\frac{\alpha}{2}-1}\ \ \ \ (0\leq
\alpha<2)\,.
  \end{equation}
In other words, (\ref{eq:stri-cond}) says that
$\phi'$ varies regularly at infinity with index $\frac{\alpha}{2}-1$. 
A novelty here is the case $\alpha=0$. 
\item[(c)] The condition {\bf (A-4)} is used only to obtain Green function
estimates. 
\end{itemize}
\end{remark}

Now we state the main result of this paper. By $\mbox{diam}(D)$ we denote the
diameter of
$D$.

\begin{thm}\label{t:Gest2}
Suppose that 
$X=(X_t, \P_x:\, t\ge 0, x \in \R^d)$
 is a transient subordinate Brownian motion 
whose
characteristic exponent is given by $\Phi(\theta)=\phi(|\theta|^2)$,
$\theta\in \R^d$, with $\phi$ satisfying {\bf(A-1)}--{\bf(A-5)}.

Then for every bounded $C^{1,1}$ open set $D$ (see Definition \ref{def:c11})
in $\R^d$ with characteristics $(R, \Lambda)$,  there exists
$c=c(\text{diam}(D), R, \Lambda,\phi, d)> 1$ such that
the Green function $G_D(x,y)$ of $X$ in $D$ satisfies 
\begin{align}\label{e:Gest21-alt1}
 c^{-1}g_D(x,y)\leq G_D(x,y)\leq cg_D(x,y)
\end{align}
with 
\begin{align}\label{e:smallgest}
g_D(x,y)=\left(1 \wedge
\frac{\phi(|x-y|^{-2})}{\sqrt{\phi(\delta_D(x)^{-2})
\phi(\delta_D(y)^{-2})}}\right)\,
\frac{\phi'(|x-y|^{-2})}{|x-y|^{d+2}\phi(|x-y|^{-2})^2}\,.
\end{align}
\end{thm}

Before we discuss a  corollary of Theorem \ref{t:Gest2}, we record a simple fact. 
\begin{lemma}\label{l:psilow}
If $\delta_* \in (0,1)$ and $\psi$ is a Bernstein function satisfying  
\begin{equation}\label{eq:as-22}
  \frac{\psi(\lambda x)}{\psi(\lambda)}\geq
\sigma_*\,x^{1-\delta_*}\ \text{ for all
}\ x\geq 1\ \text{ and }\ \lambda\geq
\lambda_*\, ,
\end{equation}
for some $\sigma_*, \lambda_*>0$,  
then there exists a constant $c>0$ 
such that $\psi(\lambda) \le c \lambda
\psi'(\lambda)$ for all $\lambda\geq
\lambda_*$.
\end{lemma}
\pf
Let $a_1 = 2 \vee (\frac{2}{
\sigma_*})^{\frac{1}{1-\delta_*}}.$
Since $\psi'$ is decreasing, 
\begin{equation}\label{eq:as-23}
(a_1-1)\lambda \psi'(\lambda) \ge \int\limits_\lambda^{a_1\lambda} \psi'(t)dt =
\psi(a_1\lambda)-\psi(\lambda).\end{equation}
Let $\lambda \ge \lambda_*$. Since $\psi(a_1\lambda) \ge  \sigma_*\,a_1^{1-\delta_*}  \psi(\lambda)$ by
\eqref{eq:as-22}, 
we get from \eqref{eq:as-23}
$$
(a_1-1)\lambda \psi'(\lambda) \ge (
\sigma_* a_1^{1-\delta_*}-1)  \psi(\lambda)  \ge  \psi(\lambda).
$$
\qed

Now we consider the following  upper and lower scaling conditions on the Laplace exponent $\phi$
 with exponents in the range $(0,1)$:
 there exist constants $c_1, c_2, \lambda_1>0$, $\alpha, \beta \in (0,2)$
and
$\alpha \le \beta$ such that  
\begin{equation}\label{e:form-of-phi}
c_1\,x^{\alpha/2} \le   \frac{\phi(\lambda x)}{\phi(\lambda)}\leq
c_2\,x^{\beta/2}\ \text{ for all
}\ x\geq 1\ \text{ and }\ \lambda\geq\lambda_1\,.
\end{equation}
Define 
\begin{equation}\label{e:whg}
\wh g_D(x,y)=
\left(1 \wedge
\frac{\phi(|x-y|^{-2})}{\sqrt{\phi(\delta_D(x)^{-2})
\phi(\delta_D(y)^{-2})}}\right)\, \frac{1}{|x-y|^d\,
\phi(|x-y|^{-2})}.
\end{equation}

If 
$\phi$  is a complete Bernstein
function such that \eqref{e:form-of-phi} holds, then 
$$\liminf_{x \to  \infty} \phi (x)\ge c_1  \lambda_1^{-\alpha/2}\phi (\lambda_1)\liminf_{x \to  \infty} x^{\alpha/2}=\infty.$$
Thus {\bf(A-1)}--{\bf(A-2)} hold. Moreover, 
applying  Lemma \ref{l:psilow} and \eqref{e:Berall1} below, 
\eqref{e:form-of-phi} implies that 
$\lambda \phi'(\lambda) \le \phi(\lambda) \le c \lambda \phi'(\lambda)$ for all $\lambda\geq\lambda_0$ and so 
{\bf(A-3)} and {\bf(A-5)} hold and 
\eqref{e:Gest21-alt1} is equivalent to \eqref{e:Gest21-alt12}.
Therefore Theorem \ref{t:Gest2} 
 gives the
following extension of the main result in \cite{KSV2}.
\begin{corollary}\label{c:Gest}
Suppose that $X=(X_t:\, t\ge 0)$ is a transient subordinate Brownian motion  whose
characteristic exponent is given by $\Phi(\theta)=\phi(|\theta|^2)$,
$\theta\in \R^d$, 
where $\phi:(0,\infty)\to [0,\infty)$  is a complete Bernstein
function such that \eqref{e:form-of-phi} holds. 
We further assume that {\bf(A-4)} hold with 
$\delta_0=1-\beta/2$ 
when $d=1$.

Then for every bounded $C^{1,1}$ open set $D$
in $\R^d$ with characteristics $(R, \Lambda)$,  there exists
$c=c(\text{diam}(D), R, \Lambda,\phi, d)> 1$ such that
the Green function $G_D(x,y)$ of $X$ in $D$ satisfies the following
estimates:
\begin{equation}\label{e:Gest21-alt12}
  c^{-1}\wh g_D(x,y)\leq G_D(x,y)\leq c \wh g_D(x,y)
\end{equation}
where $\wh g_D(x,y)$ is defined in \eqref{e:whg}.
\end{corollary}

In \cite{KSV2}, the above 
result 
 is proved when, instead of
\eqref{e:form-of-phi}, 
$\phi$ satisfies
\begin{align}
\label{e:ell}
 \phi(\lambda)\asymp \lambda^{\alpha/2}\ell(\lambda),\ \lambda \to\infty\ \ \
(0<\alpha<2)\, 
\end{align}
where $\ell$ varies slowly at infinity, i.e.
\[
 \forall\, x>0\ \ \ \  \lim_{\lambda \to\infty}\frac{\ell(\lambda
x)}{\ell(\lambda)}=1\,.
\]
By Potter's theorem (see \cite[Theorem 1.5.6(i)]{BGT}), \eqref{e:ell} clearly 
implies \eqref{e:form-of-phi}.

Using Green function estimates in Theorem \ref{t:Gest2} for $d \ge 3$ and a dimension reduction argument (see the proof of Theorem \ref{UBHP}), we 
 prove the boundary Harnack principle for
subordinate Brownian motions satisfying  {\bf (A-1)}, {\bf (A-2)}, {\bf (A-3)}
and {\bf(A-5)} in $C^{1,1}$ open set. We emphasize that in the next theorem we
do not assume neither the transience nor {\bf (A-4)}. 

\begin{thm}\label{t:bhp}
Suppose that 
$X=(X_t, \P_x:\, t\ge 0, x \in \R^d)$
 is a (not necessarily transient) subordinate
Brownian motion satisfying {\bf (A-1)}, {\bf (A-2)}, {\bf (A-3)} and {\bf
(A-5)} and that $D$ is a
(possibly unbounded) $C^{1, 1}$ open set in $\bR^d$ with
characteristics $(R, \Lambda)$. Then there exists
$c=c(R, \Lambda, \phi)>0$  such that for every $r \in (0, \frac{R \wedge
1}{4}]$, $z\in \partial D$ and any nonnegative function $u$ in $\R^d$
that is harmonic in $D \cap B(z, r)$ with respect to $X$ and
vanishes continuously on $ D^c \cap B(z, r)$, we have
\begin{equation}\label{e:bhp_m}
\frac{u(x)}{u(y)}\,\le c\,
\sqrt{\frac{\phi(\delta_D(y)
^{-2})}{\phi(\delta_D(x)
^{-2})}}
\qquad \hbox{for every } x, y\in  D \cap B(z, \tfrac{r}{2}).
\end{equation}
\end{thm}

We remark that Theorem \ref{t:bhp} 
covers
the processes in Examples 1-3 without
the assumptions on transience. 

Our paper is organized as follows. In Section \ref{sec:2} we record some
preliminary results concerning subordinate Brownian motions obtained in
\cite{KM}. We start Section \ref{sec:hs} by analyzing special harmonic
functions in half-space and use these results to obtain key probabilistic
estimates on $C^{1,1}$ open sets. Section \ref{sec:pk} contains estimates of
Poisson kernel on balls which are used in Section \ref{sec:bhp} to obtain
the uniform boundary Harnack principle on arbitrary open sets. After proving
sharp Green function estimates in Lipschitz domains in Section \ref{s:3}, we
finally obtain in Section \ref{s:4} the boundary Harnack principle and sharp
Green function estimates in $C^{1,1}$ open sets.

{\bf Notation. } Throughout the paper we use the notation $f(r)\asymp
g(r),\ r\to a$ to denote that $\tfrac{f(r)}{g(r)}$ stays between two positive
constants as $r\to a$. 
We say that
$f\colon\R\rightarrow \R$ is increasing if
$s\leq t$ implies $f(s)\leq f(t)$ and analogously for a decreasing function. 
For $a, b\in \bR$, we set
$a\wedge b:=\min \{a, b\}$, $a\vee b:=\max\{a, b\}$.
For a Borel
set $A\subset \bR^d$, we also use $|A|$ to denote its Lebesgue
measure.
We will use ``$:=$" to denote a definition,
which is read as ``is defined to be".

We will use the following conventions in this paper. The values of
the constants $C_{1},C_2, C_3$, $C_4$ and $\eps_1$  will remain the same
throughout this paper, while 
%$c_1, c_2, \ldots$
$c, c_1, c_2, \ldots$
  stand for constants
whose values are unimportant and which may change from one
appearance to another. All constants are positive finite numbers.
The labeling of the constants $
%c, 
c_1, c_2, \ldots$ starts anew in the proof of
each result. 
The dependence of 
the constant $c$ 
the constants $c, c_1, c_2, \ldots$
on the dimension $d$ will
not be mentioned explicitly. 

\section{
Preliminaries}\label{sec:2}

By 
concavity, 
we see that
every Bernstein function $
\psi$ satisfies
\begin{equation}\label{e:Berall}
\psi(t\lambda)\le
\lambda
\psi(t)
\qquad \text{ for all }
\lambda \ge 1, t >0.
\end{equation}
Thus 
\begin{equation}\label{eq:decr}
\lambda \mapsto \frac{
\psi(\lambda)}{\lambda} \ \text{ is decreasing,}
\end{equation}
 which implies
\begin{equation}\label{e:Berall1}
\lambda 
\psi'(\lambda)\le
\psi(\lambda)
\qquad \text{ for all }
\lambda >0.
\end{equation}

We first  recall the following results from \cite{KM}.

\begin{lemma}{\rm \cite[Lemma 
4.1]{KM}}\label{lem:gr-10}
Suppose that $
\psi$ is a special Bernstein function, 
i.e., $\lambda \mapsto \frac{\lambda}{\psi(\lambda)}$ is also a Bernstein
function.
Then  the  functions $\eta_1,\eta_2\colon (0,\infty)\rightarrow (0,\infty)$
given by
\[
 \eta_1(\lambda)=\lambda^2
\psi'(\lambda) \ \text{ and }\
\eta_2(\lambda)=\lambda^2\frac{
\psi'(\lambda)}{
\psi(\lambda)^2}
\]
are increasing\,.
\end{lemma}

The next result is a simple consequence of Lemma \ref{lem:gr-10} and we will use 
it several times in this paper. 
\begin{corollary}\label{c:new1}
Suppose that $
\psi$ is a special Bernstein function. 
%For every $d \ge 1$, $a>1$ 
For every $d \ge 1$, $a>1$, $\lambda>0$ 
and $b \in (0,1)$
we have 
$$
ba^{-d-3}\lambda^{-d-2}\tfrac{
\psi'(\lambda^{-2})}{
\psi(\lambda^{-2})^2}\le t^{-d-2}\tfrac{
\psi'(t^{-2})}{
\psi(t^{-2})^2}\le a b^{-d-3}\lambda^{-d-2}\tfrac{
\psi'(\lambda^{-2})}{
\psi(\lambda^{-2})^2}\quad  \forall 0<b\lambda \le t \le a \lambda.
$$
\end{corollary}
\pf
We use the fact that $t \to t^{-4}\tfrac{
\psi'(t^{-2})}{
\psi(t^{-2})^2}$ is decreasing (by Lemma \ref{lem:gr-10}) and $t \to \tfrac{
\psi'(t^{-2})}{
\psi(t^{-2})^2}$ is increasing.  When $d \ge 2$ for all $0<b\lambda \le t \le a \lambda$
$$
a^{-d-2}\lambda^{-d-2}\tfrac{
\psi'(\lambda^{-2})}{
\psi(\lambda^{-2})^2}\le t^{-d-2}\tfrac{
\psi'(t^{-2})}{
\psi(t^{-2})^2}\le  b^{-d-2}\lambda^{-d-2}\tfrac{
\psi'(\lambda^{-2})}{
\psi(\lambda^{-2})^2}.
$$
If $d=1$, 
then for every $0<b\lambda \le t \le a \lambda$
 \begin{align*}
 t^{-3}\tfrac{
\psi'(t^{-2})}{
\psi(t^{-2})^2}&=t \left( t^{-4}\tfrac{
\psi'(t^{-2})}{
\psi(t^{-2})^2}\right)\le a \lambda  \left(b^{-4}\lambda^{-4}\tfrac{
\psi'((b\lambda)^{-2})}{
\psi((b\lambda)^{-2})^2}\right)\\&\le a b^{-4}\lambda^{-3}\tfrac{
\psi'((b\lambda)^{-2})}{
\psi((b\lambda)^{-2})^2}
\le a b^{-4}\lambda^{-3}\tfrac{
\psi'(\lambda^{-2})}{
\psi(\lambda^{-2})^2},
\end{align*}
and similarly
$$
t^{-3}\tfrac{
\psi'(t^{-2})}{
\psi(t^{-2})^2}\ge b a^{-4}\lambda^{-3}\tfrac{
\psi'(\lambda^{-2})}{
\psi(\lambda^{-2})^2}.
$$
\qed

Recall that we will always assume that 
the Laplace exponent $\phi$ of $S$ 
satisfies {\bf (A-1)}--{\bf (A-3)}.
We also recall the following elementary fact from \cite{KM} 
which says that 
{\bf(A-3)} controls the growth of $\phi$. 
\begin{lemma}{\rm \cite[Lemma 
3.2 (ii)]{KM}}\label{p:USC}
For every $\eps>0$ there exists $c(\eps, \sigma)>1$ such that 
\begin{equation}\label{e:USC}
  \frac{\phi(\lambda x)}{\phi(\lambda)}\leq
c\,x^{1-\delta+\eps}\ \text{ for all
}\ x\geq 1\ \text{ and }\ \lambda\geq\lambda_0\,.
\end{equation}
\end{lemma}

The analysis of 1-dimensional subordinate Brownian motions will be crucial in
%this approach. 
our approach in this paper. 
Therefore we now consider an 
one-dimensional subordinate Brownian motion
$(Z_t, \P_x)$
with the characteristic exponent $\phi(\theta^2)$, $\theta \in \R$.

Let \[\overline{Z}_t:=\sup\{0\vee Z_s:0\le s\le t\}\] be the supremum process of
$Z$ and
let $L=(L_t:\, t\ge 0)$ be a local time of $\overline{Z}-Z$ at $0$. 
The right continuous inverse $L^{-1}_t$ of $L$ is a subordinator and it
is called the ladder time process of $Z$. The process $\overline{Z}_{L^{-1}_t}$
is also
a subordinator, called the ladder height process of $Z$.
(For the basic properties of the ladder time and ladder height
processes, we refer 
the reader to \cite[Chapter 6]{Be}.)

Let $\kappa$ 
 be the Laplace exponent of the
ladder height process of $Z$.
It follows from \cite[Corollary 9.7]{Fris} that
\begin{equation}\label{e:formula4leoflh}
\kappa(\lambda)
=\exp\left\{\frac1\pi\int\limits^{\infty}_0\frac{\log(\phi(\lambda^2\theta^2))}
{1+\theta^2}d\theta \right\}\, ,
\quad \forall \lambda>0.
\end{equation}

By our 
assumptions and 
\cite[
Proposition 13.3.7]{KSV3} or \cite[Proposition 2.1]{KMR} we see that 
 the ladder height process of $Z$ has no drift and is not compound Poisson, and
so the process $Z$ does not creep upwards.
Since $Z$ is symmetric, we know that $Z$ also does not creep
downwards. 

Denote by  $V$  the potential measure of the ladder height process of
$Z$. We will slightly abuse notation and use the same letter  $V$ to denote the
renewal function of the
ladder height process of $Z$, that is  $V(t)=V((0,t))$. $V$ is a smooth function by \cite[Corollary
13.3.8]{KSV3}.

Combining \cite[
Proposition
 13.3.7]{KSV3} and  \cite[Proposition III.1]{Be} the
following result holds.
\begin{prop}\label{e:behofV}
There exists a constant $c>1$ such that for all $r>0$
$$
 \tfrac{c^{-1}}{\sqrt{\phi(r^{-2})}} \,\le\, V(r)\, \le\, 
\tfrac{c}{\sqrt{\phi(r^{-2})}}.
$$
\end{prop}

We next consider multidimensional subordinate Brownian motions.
Let $W=(W_t=(W^1_{t}, \dots,
W^d_{t}):t\ge 0)$ be a Brownian motion in $\R^d$ with
$$
\E\left[e^{i\theta\cdot(W_t-W_0)}\right]=
e^{-t|\theta|^2}, \qquad \forall\,  \theta \in \R^d, t>0\, ,
$$
and let $S$ be a subordinator independent of $W$ with Laplace exponent $\phi$.
In the remainder of this paper, we always assume that 
 $X=(X_t,\P_x)$ is a subordinate process
 defined by $X_t=W_{S_t}$. This process is a pure-jump symmetric L\' evy process with the characteristic
exponent $\Phi(\xi)=\phi(|\xi|^2)$, i.e.
$$
\E_0 \left[ e^{i \xi \cdot X_t }
\right]=e^{-t\Phi(\xi)}=
e^{-t\phi(|\xi|^2)}\,.
$$
 Moreover, $\Phi$ has the representation
 $$
\Phi(\xi)=\int\limits_{\R^d}(1-\cos(\xi\cdot y))j(|x|)\,dx$$
 with the L\'evy measure of the form $\Pi(dx)=j(|x|)\,dx$,
where
$$
 j(r)=\int_{(0,\infty)} (4\pi
t)^{-d/2}\exp{
\left(-\tfrac{r^2}{4t}\right)}\mu(dt),\, r>0.
$$

For any open set $D$, let us denote by  $\tau_D$ 
 the first exit
time of $D$, i.e. \[\tau_D=\inf\{t>0: \, X_t\notin D\}\,.\]
Using Proposition \ref{e:behofV}, the proof of the next result is the same as the one of \cite[Proposition
3.2]{KSV4}. So we skip the proof. 

\begin{lemma}\label{l:tau}
There exists $c>0$ such that for any  $r\in (0, \infty)$ and $x_0 \in \R^d$,
\begin{eqnarray*}
\E_x[\tau_{B(x_0,r)}]\le  c\, V(r) V(r-|x-x_0|)  \asymp
\tfrac{1}{\sqrt{\phi(r^{-2})\phi((r-|x-x_0|)^{-2})}}
\ \text{ for }\  x\in B(x_0, r).
\end{eqnarray*}
\end{lemma}

The process $X$ has a transition density $p(t,x,y)$ given by
\begin{equation}\label{eq:sub-6}
 p(t,x,y)=\int\limits_0^\infty (4\pi
t)^{-d/2}\exp{\left(-\tfrac{|x-y|^2}{4t}\right)}\P(S_t\in ds)\,.
\end{equation}
When $X$ is transient,  we can define the Green function (potential) by
\[
G(x,y)=g(|y-x|)=\int\limits_0^\infty p(t,x,y)\,dt\,.
\]
Note that $g$ and $j$ are decreasing.

The following result is proved in \cite{KM}. Note that there is an error in the statement in \cite[Proposition 4.2]{KM}. It is clear form the proof of \cite[Proposition 4.2]{KM} that \cite[Proposition 4.2]{KM} holds under the condition (A-1), (A-3) and (B)  in \cite{KM}.

\begin{prop}\label{prop:pot-alpha0}
Suppose $\phi$
satisfies {\bf(A-1)}--{\bf(A-4)}.  Then
 we have
\begin{align}
 j(r)\asymp r^{-d-2}\phi'(r^{-2}), \quad  r\to 0+\,. \label{prop:pot-alpha01}
\end{align}
If $X$ is transient, then
\begin{align}
 g(r)\asymp
r^{-d-2}\frac{\phi'(r^{-2})}{\phi(r^{-2})^2},\ r\to
0+\,.\label{prop:pot-alpha02}
\end{align}
\qed
\end{prop}

As a consequence of  \eqref{prop:pot-alpha01}
it follows that if $\phi$
satisfies {\bf(A-1)}--{\bf(A-4)} then
for any $K>0$, there exists $c=c(K)
>1$ such that
\begin{equation}\label{H:1}
j(r)\le c\, j(2r), \qquad \forall r\in (0, K).
\end{equation}
Since $\phi$ is a complete Bernstein function,  there exists a constant $c>0$ such that
$
\mu(t)\leq c\,\mu(t+1)$ for all  $t\geq 1$
(see  \cite[Lemma 2.1]{KSV2}). Thus, using this and \cite[Proposition 3.3]{KM}, 
by the proof of \cite[Proposition 13.3.5] {KSV3}
we see that the function $j$ also enjoys the following property:
 if $\phi$
satisfies {\bf(A-1)}--{\bf(A-4)} then there is a  constant $c>0$
such that
\begin{equation}\label{eq:sub-11}
 j(r+1)\leq j(r)\leq cj(r+1)\ \text{ for all } \ r\geq 1\,.
\end{equation}

Let $D\subset \R^d$ be an open subset. The killed process $X^D$ is defined by
\[X^D_t=X_t\ \text{ if }\ t<\tau_D\ \ \ \text{ and  }\ \ \  X^D_t=\Delta\ \text{
otherwise},\] where $\Delta$ is an extra
point adjoined to $D$ (usually called cemetery).

The transition density of $X^D$ is given by
\[
 p_D(t,x,y)=p(t,x,y)-\E_x\left[p(t-\tau_D,X_{\tau_D},y); \tau_D<t\right]
\]

A subset $D$ of $\R^d$ is said to be Greenian (for $X$) if $X^{D}$ is transient. When
$d\ge 3$, any non-empty open set $D\subset\bR^d$ is
Greenian. An open set $D\subset \bR^d$ is Greenian if
and only if $D^c$
 is non-polar for $X$ (or equivalently, has positive capacity with respect to $X$).
For any Greeninan open  set $D$ in $\R^d$ let
$G_D(x,y)=\int\limits_0^{\infty}p_D(t,x,y)\,dt$
be the Green function of $X^D$.
$G_D(x,y)$ is symmetric and, for fixed $y\in D$, $G_D(\cdot,y)$ is 
harmonic (with respect to $X$) in $D\setminus\{y\}$. 

The next two results are the key estimates in \cite{KM}.

\begin{prop}\label{p:green}
Suppose $X$ is transient and $\phi$
satisfies {\bf(A-1)}--{\bf(A-4)}. There exist constants $c_1,c_2>0$ and
$b_1,b_2\in (0,\frac{1}{2})$, $2b_1<b_2$
such that for all
$x_0\in \R^d$ and $r\in (0,1)$ we have 
%\[
%	c_1 \tfrac{r^{-d-2}\phi'(r^{-2})}{\phi(r^{-2})}\,\E_y\tau_{B(x_0,r)}\leq
%G_{B(x_0,r)}(x,y)\leq c_2
%\tfrac{r^{-d-2}\phi'(r^{-2})}{\phi(r^{-2})}\,\E_y\tau_{B(x_0,r)}\ 
%\]
\begin{equation}\label{e:neweest}
	c_1 \tfrac{r^{-d-2}\phi'(r^{-2})}{\phi(r^{-2})}\,\E_y\tau_{B(x_0,r)}\leq
G_{B(x_0,r)}(x,y)\leq c_2
\tfrac{r^{-d-2}\phi'(r^{-2})}{\phi(r^{-2})}\,\E_y\tau_{B(x_0,r)}\ 
\end{equation}
for all $x\in B(x_0,b_1r)$ and 
$y\in B(x_0,r) \setminus B(x_0,b_2r)$.
\end{prop}

\begin{prop}\label{p:exittime}
Suppose $X$ is transient and $\phi$
satisfies {\bf(A-1)}--{\bf(A-4)}. There exist constants $c_1>0$ and
$a\in (0,\frac{1}{3})$ so that for 
$x_0\in \R^d$ and $r\in (0,1)$ we have 
\[
\E_x[\tau_{B(x_0,r)}]\geq \tfrac{c_1}{\phi(r^{-2})}\quad \text{ for any } \ x\in
B(x_0,ar)\,.
\]
\end{prop}

Before we state the Harnack inequality,
we recall the definition of harmonic functions.

\begin{defn}\label{def:har1}
Let $D$ be an open subset of $\R^d$.
A function $u$ defined on $\R^d$ is said to be

\noindent
(i)
  harmonic in $D$ with respect to $X$ if
\[\E_x\left[|u(X_{\tau_{B}})|\right] <\infty\ \ \text{ and }\ \ 
u(x)= \E_x\left[u(X_{\tau_{B}})\right],\ 
 x\in B\,,\]
for every open set $B$ whose closure is a compact
subset of $D$;

\noindent
(ii)
regular harmonic in $D$ with respect to $X$ if
it is harmonic in $D$ with respect to $X$ and
\[u(x)= \E_x\left[u(X_{\tau_{D}})\right]\ \text{ for any }\ x\in D\,.\]
\end{defn}

The following Harnack inequality is the main result of \cite{KM}.
\begin{thm}[Harnack inequality]\label{T:Har}
Suppose that  $\phi$ 
satisfies {\bf (A-1)}--{\bf (A-3)}.
		There exists a constant $c>0$  such that for
all $x_0\in \R^d$ and $r\in (0,1)$ we have
	\[
		h(x_1)\leq c\, h(x_2)\ \text{ for all }\ x_1,x_2\in
B(x_0,r/2)
	\]
	and for every  non-negative function $h\colon \R^d\rightarrow
[0,\infty)$ which
is harmonic in $B(x_0,r)$.
\end{thm}

%\begin{remark}\label{rem:green}
%Using Theorem \ref{T:Har} and the standard chain argument, it can be proved that under the assumptions of Proposition \ref{p:green} there exist constants $c_1',c_2'>0$ so that for any $r\in (0,1)$ and $x_0\in \R^d$
%\[
%	c_1' \tfrac{r^{-d-2}\phi'(r^{-2})}{\phi(r^{-2})}\,\E_y\tau_{B(x_0,r)}\leq
%G_{B(x_0,r)}(x,y)\leq c_2'
%\tfrac{r^{-d-2}\phi'(r^{-2})}{\phi(r^{-2})}\,\E_y\tau_{B(x_0,r)}\ 
%\]
%for all $x\in B(x_0,r/2)$ and 
%$y\in B(x_0,r) \setminus B(x_0,3r/4)$.
%\end{remark}

Using Theorem \ref{T:Har} and the standard chain argument to \eqref{e:neweest}, we have 
\begin{corollary}\label{rem:green}
Under the assumptions of Proposition \ref{p:green} there exist constants $c_1,c_2>0$ so that for any $r\in (0,1)$ and $x_0\in \R^d$
\[
	c_1 \tfrac{r^{-d-2}\phi'(r^{-2})}{\phi(r^{-2})}\,\E_y\tau_{B(x_0,r)}\leq
G_{B(x_0,r)}(x,y)\leq c_2
\tfrac{r^{-d-2}\phi'(r^{-2})}{\phi(r^{-2})}\,\E_y\tau_{B(x_0,r)}\ 
\]
for all $x\in B(x_0,r/2)$ and 
$y\in B(x_0,r) \setminus B(x_0,3r/4)$.
\end{corollary}

By the result of Ikeda and Watanabe (see \cite[Theorem 1]{IW}) the following
formula is true
\begin{equation}\label{eq:sub-105}
 \P_x(X_{\tau_D}\in F)=\int\limits_F\int\limits_DG_D(x,y)j(|z-y|)\,dy\,dz
\end{equation}
for any $F\subset \overline{D}^c$. 
We define the Poisson kernel of the set $D$ by
\begin{equation}\label{eq:sub-10} 
 K_D(x,z)=\int\limits_D G_D(x,y)j(|z-y|)\,dy,
\end{equation}
so that  $\P_x(X_{\tau_D}\in F)=\int\limits_F K_D(x,z)\,dz$ for any $F\subset
\overline{D}^c$.

\begin{prop}\label{p:Poisson1}
Suppose $X$ is transient and $\phi$
satisfies {\bf(A-1)}--{\bf(A-4)}.
There exists $c_1=c_1
( \phi)>0$ and $c_2=c_2
( \phi)>0$ such that for every $r \in (0, 1)$ and $x_0 \in \R^d$,
\begin{eqnarray}
K_{B(x_0,r)}(x,y) \,&\le &\, c_1 \,
\tfrac{j(|y-x_0|-r)}{\sqrt{\phi(r^{-2})\phi((r-|x-x_0|)^{-2})}}\label{P1}\\
 &\le &\, c_1 \, \tfrac{j(|y-x_0|-r)}{ \phi(r^{-2})}\ \label{P1-worse}
\end{eqnarray}
for all $(x,y) \in B(x_0,r)\times \overline{B(x_0,r)}^c$ and
\begin{equation}\label{P2}
K_{B(x_0, r)}(x_0, y) \,\ge\, c_2\, \tfrac{j(|y-x_0|)}{
\phi(r^{-2})} \qquad \textrm{ for all } y \in \overline{B(x_0, r)}^c.
\end{equation}
\end{prop}

\pf
First using \eqref{H:1} and \eqref{eq:sub-11} to \eqref{eq:sub-10}, then applying 
Lemma \ref{l:tau} and Proposition \ref{p:exittime},
\eqref{P1} and \eqref{P2} follow easily (see the proof of \cite[Proposition 13.4.10]{KSV3} for the
details). 
\eqref{P1-worse} follows from \eqref{P1} and the fact that $\phi$ is
increasing.
\qed

\section{Analysis on half-space and $C^{1,1}$ open sets}\label{sec:hs}

In this section we
establish key estimates which will be used in sections
later in this paper. 

Recall that  $X=(X_t:\, t\ge 0)$ is the $d$-dimensional
subordinate
Brownian motion defined by $X_t=W_{S_t}$ where $W=(W^1,\dots, W^d)$
is a (not necessarily transient) $d$-dimensional Brownian motion and $S=(S_t:\,
t\ge 0)$ an
independent subordinator with the Laplace exponent $\phi$
satisfying {\bf(A-1)}-{\bf(A-3)}.  In this section, we further assume that {\bf(A-4)} holds.

 Let $Z=(Z_t:\, t\ge 0)$ be
the one-dimensional subordinate Brownian motion defined 
by
$Z_t:=W^d_{S_t}$.

Recall that $V$ 
denotes the renewal function of the ladder height process of
$Z$.
We use the notation \[\bR^d_+:=\{
x=(x_1, \dots, x_{d-1}, x_d):=(\tilde{x}, x_d)  \in \bR^d: x_d > 0
\}\] for the half-space. 

Set $w(x):=V((x_d)^+)$.
Since $Z_t=W^d_{S_t}$ has a transition density,
by using \cite[Theorem 2]{Sil},
the proof of the next result is the same as the one of \cite[Theorem
4.1]{KSV2}. We omit the proof. 

\begin{thm}\label{t:Sil}
The function $w$ is harmonic in $\R^d_+$ with respect to $X$
and, for any $r>0$, regular harmonic in
$\R^{d-1}\times (0, r)$ 
for $X$.
\end{thm}

Using  Theorem \ref{t:Sil}, \eqref{H:1} and \eqref{eq:sub-11}, 
the proof of the next result is the same as the one of \cite[Proposition
3.3]{KSV5}.

\begin{prop}
\label{c:cforI}
For all positive constants
$r_0$ and $L$, we have
$$
\sup_{x \in \R^d:\, 0<x_d <L} \int\limits_{B(x, r_0)^c \cap \bR^d_+}
w(y) j(|x-y|)\, dy < \infty\, .
$$
\end{prop}
Define an operator
($\sA$, $\mathfrak{D}(\sA)$)
by
\begin{align}
 \sA  f(x)&:= \mathrm{p.v.} \int\limits_{\R^d}
\left(f(y)-f(x)\right)j(|y-x|)\, dy\nn\\&:=\lim_{\eps \downarrow 0}
\int\limits_{\{y\in \bR^d: |x-y| > \eps\}}
\left(f(y)-f(x)\right)j(|y-x|)\, dy\, \nn\\
\mathfrak{D}(\sA)\nn
&:=\left\{f:\R^d\to \R: \lim_{\eps \downarrow 0}
\int\limits_{\{y\in \bR^d: |x-y| > \eps\}}
\left(f(y)-f(x)\right)j(|y-x|)\, dy \right.\\&\qquad\qquad\left.\text{ exists
and it is finite } \right\}.
\label{generator}
\end{align}

Let $C^2_0$ be the collection of $C^2$
functions in
$\RR^d$ vanishing at infinity. 
 It is well known that
$C^2_0\subset \mathfrak{D}(\sA)$ and
that by the rotational symmetry of $X$, $\sA$ restricted to $C^2_0$
coincides with the infinitesimal generator $\sL$ of the process $X$ (see e.g.
\cite[Theorem 31.5]{S}).

Since $V$ is smooth by \cite[Corollary 13.3.8]{KSV3}, using our Theorem \ref{t:Sil},
\eqref{H:1} and 
\eqref{eq:sub-11},  
the proof of the next result is the same as \cite[Proposition
3.3]{KSV5} or \cite[Proposition
4.2]{KSV2}, so we skip the proof.

\begin{thm}\label{c:Aw=0}
$\sA w(x)$ is well defined and $\sA w(x)=0$ for all $x \in \bR^d_+$.
\end{thm}

In the rest of this section we aim to prove two key estimates of the exit
probability
and the exit time for $C^{1,1}$ open sets. Let us recall the definition of 
a
$C^{1,1}$ open set. 
\begin{defn}\label{def:c11}
An open set $D$ in $\bR^d$ ($d\ge 2$) is said to be
a $C^{1,1}$ open set if there exist a localization radius $R>0$ and
a constant $\Lambda>0$ such that for every $z\in\partial D$, there
exist a $C^{1,1}$-function $\psi=\psi_z: \bR^{d-1}\to \bR$
satisfying $\psi (0)=0$,  $\nabla\psi (0)=(0, \dots, 0)$, \[\| \nabla
\psi \|_\infty \leq \Lambda,\ \ \ \ | \nabla \psi (x)-\nabla \psi (w)|
\leq \Lambda |x-w|, \quad x,w \in \R^{d-1}\] and an orthonormal coordinate system $CS_z$:
$y=(y_1, \cdots, y_{d-1}, y_d):=(\wt y, \, y_d)$ with origin at $z$
such that
$$
B(z,R)\cap D=\{ y=(\wt y, \, y_d)\in B(0, R) \mbox{ in } CS_z: y_d
>\psi (\wt y) \}.
$$
The pair $(R, \Lambda)$ is called the characteristics of the
$C^{1,1}$ open set $D$. 
By a $C^{1,1}$ open set in $\bR$ we mean an open
set which can be expressed as the union of disjoint intervals so
that the minimum of the lengths of all these intervals is positive
and the minimum of the distances between these intervals is
positive.
\end{defn}
\begin{remark}
 In some literature, the
$C^{1,1}$ open set defined above is called a {\it uniform} $C^{1,1}$
open set since $(R, \Lambda)$ is universal for all $z\in \partial D$.
\end{remark}

For $x\in \bR^d$, let $\delta_{\partial D}(x)$ denote the Euclidean
distance between $x$ and $\partial D$.
Recall that for 
any
$x\in \bR^d$, $\delta_{ D}(x)$ is the Euclidean
distance between $x$ and $D^c$.
It is well known that any
$C^{1, 1}$ open set $D$ with characteristics $(R, \Lambda)$ there exists 
$r_1>0$
so that 
the following holds true:
\begin{itemize}
	\item[(i)] {\it uniform interior ball condition}, i.e.
	 for every $x\in D$
with $\delta_{D}(x)< r_1$ there exists $z_x\in \partial D$ so that
\[|x-z_x|=\delta_{\partial D
}(x)\ \ \text{ and }\ \ 
B(x_0, r_1)\subset D,\]
for $x_0=z_x+r_1\frac{x-z_x}{|x-z_x|}$ ;
	\item[(ii)] {\it uniform exterior ball condition}, i.e.
	 for every $y\in \R^d\setminus D$
with $\delta_{
\partial D}(y)< r_1$ there exists $z_y\in \partial D$ so that
\[|y-z_y|=\delta_{\partial D
}(y)\ \ \text{ and }\ \ 
B(y_0, r_1)\subset \R^d\setminus D,\]
for $y_0=z_y+r_1\frac{y-z_y}{|y-z_y|}$.
\end{itemize}

Assume for the rest of this section that  $D$ is a $C^{1,1}$ open
set with characteristics $(R, \Lambda)$
satisfying the  uniform
interior ball condition and the uniform exterior ball condition with
the radius $R\leq 1$ (by choosing $R$ smaller if necessary).

Before we prove our technical Lemma \ref{L:Main} below, we need some preparation. 
\begin{lemma}\label{l:new}
Under the assumption {\bf(A-5)},
 if $d \ge 2$ and the constant  $\delta$ in {\bf(A-3)}  satisfies  $0<\delta \le 
\frac{1}{2}$, then for every $M>0$
$$
\sup_{x \in [0, M/4]}\int\limits_{0}^{M}  v(s/6)\left(
\phi(
|s-
x|^{-2}) 
|s- x| +
\int\limits_{|s-
x|}^{ {M}} \phi(r^{-2}) 
 dr
\right) ds =c(M, \phi)< \infty.$$
\end{lemma}

\pf
Let
$$
I:=\int\limits_0^{ x/2} 
  v(s/6)\left(
\phi(
|s-
x|^{-2}) 
|s- x| +
\int\limits_{|s-
x|}^{ {M}} \phi(r^{-2}) 
 dr
\right) ds
$$
$$
II:= \int\limits_{x/2}^{2x}
  v(s/6)\left(
\phi(
|s-
x|^{-2}) 
|s- x| +
\int\limits_{|s-
x|}^{ {M}} \phi(r^{-2}) 
 dr
\right) ds
$$
and 
$$
III:= \int\limits_{2x}^{ M}
  v(s/6)\left(
\phi(
|s-
x|^{-2}) 
|s- x| +
\int\limits_{|s-
x|}^{ {M}} \phi(r^{-2}) 
 dr
\right) ds.
$$
We consider these three parts separately. 

First, for $s\in (0, x/2)$, we have
$x \ge x-s=|x-s|
\ge x/2$. Thus  
using  \eqref{e:Berall} and Proposition \ref{e:behofV},
\begin{align*}
I\le&  x \phi(4 x^{-2}) \int\limits_0^{ x/2}
v(s/6)   ds+\int\limits_0^{ x/2} 
  v(s/6)ds 
\int\limits_{x/2}^{ {M}} \phi(
r^{-2}) 
 dr
 \\
\le& 6\left(4  x \phi(x^{-2})+ \int\limits_{x/2}^{ {M}} \phi(
r^{-2}) 
 dr \right)  V(x/12) \le c_{1}x \phi(x^{-2})^{1/2}+c_1\int\limits_{x/2}^{ {M}} \tfrac{\phi(r^{-2})}{ \phi( x^{-2})^{1/2}}
 dr.
\end{align*}
The first term is finite by Lemma \ref{p:USC}. Also by Lemma \ref{p:USC} with $\eps=\delta/2$, 
\begin{align*}
\int\limits_{x/2}^{ {M}} \tfrac{\phi(r^{-2})}{ \phi( x^{-2})^{1/2}}
 dr
 \le c_2 \int\limits_{x/2}^{ {M}} \phi(r^{-2})^{1/2} \le c_3 \int\limits_{x/2}^{ M}r^{-(1-\delta+\eps)}dr
 \le c_4 M^{\delta -\eps} <\infty.
\end{align*}

Applying Proposition \ref{e:behofV}, we deduce
\begin{eqnarray}
v(s/6) \le \frac6{s} \int\limits_{0}^{s/6}v(t)dt =  \frac6{s} V(s/6) \le
c_{5}
\frac1{s} \phi(s^{-2})^{-1/2}, \quad 
\text{ for all }\ s>0.
\label{e:nn5}
\end{eqnarray}
By  \eqref{e:Berall} and  \eqref{e:nn5},
\begin{align*}
II\le &c_5x^{-1}  \phi( x^{-2})^{-1/2}  \int\limits_{x/2}^{2x}
 \phi(|s- x|^{-2}) |s-x| ds+c_5 x^{-1}  \phi( x^{-2})^{-1/2}  \int\limits_{x/2}^{2x}
\int\limits_{|s-x|}^{ {M}} \phi(r^{-2})
 dr ds
\\
\le &  c_{6}x^{-1} \phi( x^{-2})^{1/2} \int\limits_{0}^{x}
 t  \tfrac{\phi(t^{-2})}{\phi(x^{-2})}  dt \\
 &\quad + c_5x^{-1}    \int\limits_{x/2}^{2x} \tfrac{\phi(|s-x|^{-2})^{1/2}}{\phi( x^{-2})^{1/2}}
\int\limits_{|s-x|}^{ {M}} \tfrac{\phi(r^{-2})^{1/2}}{\phi(|s-x|^{-2})^{1/2}}\phi(r^{-2})^{1/2}
 dr
 ds.
 \end{align*}
 Applying  Lemma \ref{p:USC} twice  with $\eps=\delta/2$ 
 to $\tfrac{\phi(t^{-2})}{\phi(x^{-2})}$
 and $\phi( x^{-2})^{1/2}$, we get  
\begin{align*}& x^{-1} \phi( x^{-2})^{1/2} \int\limits_{0}^{x}
 t  \tfrac{\phi(t^{-2})}{\phi(x^{-2})}  dt 
 \le 
  c_{7} x^{-1} x^{\delta-\eps-1} \int\limits_{0}^{x}
 t  \left(\tfrac{t}{x}\right)^{-2+2(\delta- \eps)}  dt \\
&=  c_{7}  x^{ -(\delta- \eps)}   \int\limits_{0}^{x}   t^{-1 +2(\delta- \eps)}    dt 
 \le c_{8}    x^{\delta- \eps}\le c_{8}    M^{\delta- \eps} < \infty. 
<\infty
 \end{align*}
 On the other hand, since $|s-x| \le 3x$ 
 for $s \le 2x$, 
  \eqref{e:Berall}, Lemma \ref{p:USC} with $\eps=\delta/2$ and {\bf (A-5)}
  imply 
 \begin{align*}
   & x^{-1}    \int\limits_{x/2}^{2x} \tfrac{\phi(|s-x|^{-2})^{1/2}}{\phi( x^{-2})^{1/2}}
\left(
\int\limits_{|s-x|}^{ {M}} \tfrac{\phi(r^{-2})^{1/2}}{\phi(|s-x|^{-2})^{1/2}}\phi(r^{-2})^{1/2}
 dr
\right) ds\\
&\le   c_9 x^{-1}    \int\limits_{x/2}^{2x}  \tfrac{x^{1-\delta+\eps} }{|s-x|^{1-\delta+\eps}}
\left(
\int\limits_{|s-x|}^{ {M}}  (\tfrac{|s-x|}{r})^{1-\delta_1} r^{-(1-\delta+\eps)}
dr
\right) ds\\
&=   c_9 x^{-\delta+\eps}    \int\limits_{x/2}^{2x}   |s-x|^{-\delta_1+\delta-\eps}
\left(
\int\limits_{|s-x|}^{ {M}}  r^{-2+\delta_1+\delta-\eps}
dr
\right) ds\\
&\le    c_{10} x^{-\delta+\eps}    \int\limits_0^{x}   t^{-\delta_1+\delta-\eps}
\left(
\int\limits_{t}^{ {M}}  r^{-2+\delta_1+\delta-\eps}
dr
\right) dt=:A\\
\end{align*}
If $2-\delta_1-\delta+\eps >1$, 
\begin{align*}
&A \le  c_{11} x^{-\delta+\eps}    \int\limits_0^{x}   
  t^{-1+2(\delta-\eps)}ds\le   
  c_{12} x^{\delta-\eps} \le  c_{12} M^{\delta-\eps} <\infty.
\end{align*}
If $2-\delta_1-\delta+\eps =1$,  
integration by parts yields
\begin{align*}
&A \le c_{13} x^{-\delta+\eps}    \int\limits_0^{x}   t^{-\delta_1+\delta-\eps}
\ln(M/t) dt \le 
c_{14} x^{-\delta+\eps}x^{1-\delta_1+\delta-\eps}
\ln(M/x)\\
& \le c_{14} \sup_{x \in [0, M/4]} x^{1-\delta_1}
\ln(M/x) 
 <\infty.
\end{align*}
If $2-\delta_1-\delta+\eps <1$, 
\begin{align*}
A \le  c_{10} x^{-\delta+\eps}    \int\limits_0^{x}   t^{-\delta_1+\delta-\eps}
\left(
\int\limits_{0}^{ {M}}   r^{-2+\delta_1+\delta-\eps}
dr
\right) dt\le   c_{15}      x^{1-\delta_1}\le   c_{15}      M^{1-\delta_1}
  <\infty.
\end{align*}
Thus
$
II   < \infty.
$

For $III$, we note that 
$s \ge s-x=|s-x|\geq
s/2$ for $s\geq 2x$.
Using this, \eqref{e:Berall},   \eqref{e:nn5}, Lemma \ref{p:USC}  with $\eps=\delta/2$
and {\bf (A-5)}, we get 
\begin{align*}
&III\le \int\limits_{2x}^{ M}
v(s/6) s \phi(4s^{-2})  ds+\int\limits_{2x}^{ M}
 v(s/6)
\int\limits_{s/2}^{ {M}} \phi(
r^{-2}) 
 dr
 ds\\
& \le c_{16} \int\limits_{2x}^{ M} \phi(s^{-2})^{1/2}  ds+ c_{16}\int\limits_{2x}^{ M} s^{-1} \int\limits_{s/2}^{ {M}}
\tfrac{\phi(r^{-2})^{1/2}}{\phi( s^{-2})^{1/2}}
{\phi(r^{-2})^{1/2}}
 drds\\
&\le   c_{17} \int\limits_{0}^{ M} 
s^{-1+(\delta-\eps)}ds +c_{17}\int\limits_{2x}^{ M} s^{-1} \int\limits_{s/2}^{ {M}}
(s/r)^{1-\delta_1}  \phi(r^{-2})^{1/2} drds.
\end{align*}
Clearly the first term is finite. Using Lemma \ref{p:USC} with $\eps=\delta/2$,  the second term is bounded by
\begin{align*}
B:= c_{18}\int\limits_{2x}^{ M} s^{-\delta_1}\int\limits_{s/2}^{ {M}} r^{-2+\delta_1+\delta-\eps}drds .
\end{align*}
Thus if $2-\delta_1-\delta+\eps >1$,
\begin{align*}
&B  \le 
c_{19} \int\limits_{2x}^{ M} s^{-(1-\delta+\eps)}ds\le 
c_{19} \int\limits_{0}^{ M} s^{-(1-\delta+\eps)}ds  <\infty.
\end{align*}
If $2-\delta_1-\delta+\eps =1$,  
using integration by parts we obtain
\begin{align*}
&B \le 
c_{20} \int\limits_{2x}^{ M} s^{-\delta_1} \ln (M/s)
ds
\le 
c_{21}  x^{1-\delta_1} \ln (M/x)\le 
c_{21}  \sup_{x \in [0, M/4] } x^{1-\delta_1} \ln (M/x)
 <\infty.
\end{align*}
Finally, If $2-\delta_1-\delta+\eps <1$, 
\begin{align*}
B \le \
 c_{18}\int\limits_{2x}^{ M} s^{-\delta_1}\int\limits_{0}^{ {M}} r^{-(2-\delta_1-\delta+\eps)}drds 
 \le 
c_{22} \int\limits_{2x}^{ M} s^{-\delta_1} ds\le c_{22} \int\limits_{0}^{ M} s^{-\delta_1} ds
 <\infty.
\end{align*}
Thus $III<\infty$ and so we have proved the lemma.
\qed

\begin{lemma}\label{L:Main} Assume additionally that  {\bf(A-5)} holds. 
Fix $Q \in \partial D$  and let 
\[
h(y)=
 \begin{cases}
  V(\delta_D(y)) & y\in 
  B(Q,R)\cap D\\
0 & \text{otherwise}\,.
 \end{cases}
\]
There exists
$
C_1=
C_1(\Lambda, R, \phi)>0$ independent of
the point $Q \in \partial D$ such that $\sA h$ is well defined in $D\cap B(Q,
\frac{R}{4})$ and
\bee\label{e:h3}
|\sA h(x)|\le
C_1 \quad \text{ for all } x \in D\cap B(Q, \tfrac{R}{4})\, .
\eee
\end{lemma}

\pf
We first note that when $d=1$, the lemma follows from Proposition
\ref{c:cforI} and Theorem \ref{c:Aw=0} by following the same proof as the one in
\cite[Lemma 4.4]{KSV2}.

Assume now that $d\geq 2$. 
Fix $x\in D\cap B(Q,\frac{R}{4})$ and let $x_0 \in \partial D$ such that 
$\delta_D(x)=|x-x_0|$.

Denote by $\psi$ a $C^{1,1}$ function and by $CS=CS_{x_0}$ an orthonormal
coordinate system with $x_0$ chosen so that $x=(\wt 0,x_d)$ and
\[
	B(x_0,R)\cap D=\{y=(\wt y,y_d) \ \text{ in }\ CS\colon y\in 
B(0,R),\
y_d>\psi(\wt y)\}\,.
\]

We fix such $\psi$ and the coordinate system $CS$.

Define two auxiliary functions $\psi_1,\psi_2\colon B(\wt 0,R)\rightarrow \R$ by
\[
	\psi_1(\wt y)=R-\sqrt{R^2-|\wt y|^2}\ \ \text{ and }\ \ \psi_2(\wt
y)=-\left(R-\sqrt{R^2-|\wt y|^2}\right)\,.
\]

By the interior/exterior uniform ball conditions (with radius $R$) it follows
that
\begin{equation}\label{eq:intext_1}
\psi_2(\wt y)\leq \psi(\wt y)\leq \psi_1(\wt y)\ \ \text{ for any }\ \ y\in
D\cap B(x,\tfrac{R}{4})\,.
\end{equation}

Now we define a function $h_x(y)=V(\delta_{H^+}(y))$, where 
\[
	H^+=\{y=(\wt y,y_d)\ \text{ in }\ CS\colon y_d>0\}
\]
denote the half-space in $CS$. 

Since $\delta_{H^+}(y)=(y_d)^+$ in $CS$, we can use Theorem \ref{c:Aw=0} to
deduce that
\[
	\sA  h_{x}(y)=0, \quad  \forall y\in H^+.
\]

Now the idea is to show that $\sA (h-h_x)(x)$ is well defined and that there
exists a constant $C_1=C_1(\Lambda,R,\phi)>0$ so that
\begin{equation}\label{eq:est_epsduh}
\int\limits_{\{y\in D\cup H^+\colon |y-x|>\varepsilon\}}
|h(y)-h_x(y)|j(|y-x|)\,dy\leq C_1\ \ \text{ for any}\ \ \varepsilon>0\,.
\end{equation}

To do this we estimate the integral in (\ref{eq:est_epsduh}) by the sum of the
following three integrals:
\begin{align*}
	I_1 & = \int\limits_{B(x,\frac{R}{4})^c}(h(y)+h_x(y))j(|y-x|)\,dy\\
	I_2 & = \int\limits_{A}(h(y)+h_x(y))j(|y-x|)\,dy,\ \text{ where }\\ &
\qquad\qquad \qquad\qquad \qquad\qquad A:=\{y\in (D\cup H^+)\cap
B(x,\tfrac{R}{4})\colon \psi_2(\wt y)\leq y_d\leq \psi_1(\wt y)\}\\
	I_3 & = \int\limits_{E}|h(y)-h_x(y)|j(|y-x|)\,dy,\ \text{ where
}E:=\{y\in B(x,\tfrac{R}{4})\colon y_d>\psi_1(\wt y)\}\,
\end{align*}
and prove that $I_1+I_2+I_3\leq C_1$ .

To estimate $I_1$ note that, by definition of $h$, $h=0$ on $B(Q,R)^c$ which
gives
\[
	I_1\leq \sup_{\scriptsize\begin{array}{c}z\in
\R^d\\0<z_d<R\end{array}}\int\limits_{
B(z,\frac{R}{4})^c\cap H^+}V(y_d)j(|z-y|)\,
dy+c_1\int\limits_{B(0,\frac{R}{4})^c} j(|y|) dy<\infty.
\] 
Here we have used Proposition \ref{c:cforI} and the fact that the L\' evy
measure is a finite measure away from the origin.

Now we estimate $I_2$. 
Denoting by $m_{d-1}(dy)$ the surface measure, we obtain
\[
	I_2\leq \int\limits_0^\frac{R}{4}\int\limits_{|\wt
y|=r}\mathbf{1}_{A}(y)(h_x(y)+h(
y))j\left(\sqrt{r^2+|y_d-x_d|^2}\right)\,m_{d-1}
(dy)\,dr\,.
\]

Since $V$ is increasing and 
\[
	R-\sqrt{R^2-|\wt y|^2}\leq \tfrac{|\wt y|^2}{R}\leq |\wt y|,
\]
we can use (\ref{eq:intext_1}) to deduce
\[
	h_x(y)+h(y)\leq 2V\left(\psi_1(\wt y)-\psi_2(\wt y)\right)\leq 2V(2|\wt
y|)\,.
\]
Then, by the fact that $j$ decreases, Proposition \ref{e:behofV} and 
\eqref{prop:pot-alpha01}, we get
\begin{align*}
I_2 & \leq 2
\int\limits_0^\frac{R}{4} \int\limits_{|\wt
y|=r}  \mathbf{1}_A(y)V(2|\wt
y|)j(r)\,m_{d-1}(dy)\,dr \\
&\leq
c_2\int\limits_0^\frac{R}{4}r^{-d-2}\frac{\phi'(r^{-2})}{\sqrt{\phi(r^{-2})}}\,
m_{d-1}(\{y\in A\colon |\wt y|=r\})\,dr\,.
\end{align*}

Noting that $|\psi_2(\wt y)-\psi_1(\wt y)|\leq \frac{2|\wt
y|^2}{R}=\frac{2r^2}{R}$ for $|\wt y|=r$, we obtain
\[
	m_{d-1}(\{y\colon |\wt y|=r,\ \psi_2(\wt y) \le y_d
	 \le \psi_1(\wt y)\})\leq c_3
r^d\ \ \text{ for }\ \ r\leq \tfrac{R}{4}\,.
\]

Thus, by the previous observation and the integration by parts we get
\begin{align*}
	I_2&\leq
c_4\int\limits_0^
\frac{R}{4}r^{-2}\frac{\phi'(r^{-2})}{\sqrt{\phi(r^{-2})}}\,
dr\,=\,c_4\int\limits_0^\frac{R}{4} r \left(- \sqrt{\phi(r^{-2})} \right)' dr \\
	&\leq c_4\left[\lim_{r\downarrow 0}
r\sqrt{\phi(r^{-2})}+\int\limits_0^\frac{R}{4}\sqrt{\phi(r^{-2})}\,dr\right]\,.
\end{align*}

By Lemma \ref{p:USC} applied to a fixed  $\varepsilon<\delta$ we see that there
is a constant $c_5=c_5(\varepsilon)>0$ so that 
\[
	\phi(r^{-2})\leq c_5r^{-2(1-\delta+\varepsilon)},
\]
which gives
\[
	I_2\leq c_4\int\limits_0^\frac{R}{4}\sqrt{\phi(r^{-2})}\,dr\leq
c_4 \sqrt{c_5}\int\limits_{0}^\frac{R}{4}r^{-1+\delta-\varepsilon}\,dr<\infty\,.
\]

In order to estimate $I_3$, we consider two cases. First, 
if
$0 <y_d=\delta_{_{H^+}}({y}) \le  \delta_D({y})$,
\begin{align}
h(y)-h_{x}(y) \le V(y_d+R^{-1}|\wt y|^2) -V(y_d) =
\int\limits_{y_d}^{y_d+R^{-1}|\wt y|^2} v(z)dz \le R^{-1}|\wt y|^2
v(y_d),  \label{e:KG}
\end{align}
since $v$ is decreasing. 

If $y_d=\delta_{_{H^+}}({y}) >  \delta_D({y})$ and $y \in E$,
using the fact that
$\delta_D({y})$ is greater than or equal to the distance between $y$
and the graph of $\psi_1$ and
 \begin{align*}
y_d-R+\sqrt{ |\wt y|^2+(R-y_d)^2} &= \tfrac{|\wt y|^2}
{\sqrt{ |\wt y|^2+(R-y_d)^2} + (R-y_d)}\,\le\, \tfrac{ |\wt y|^2} {2 (R-y_d)}
\le
\tfrac{ |\wt y|^2} {R},
\end{align*}
we 
obtain
 \begin{align}
h_{x}(y)-h(y)
\le\int\limits^{y_d}_{R-\sqrt{ |\wt y|^2+(R-y_d)^2}} v(z)dz\le  R^{-1}  |\wt
y|^2
\,v\left(R-\sqrt{ |\wt y|^2+(R-y_d)^2}\right). \label{e:KG2}
\end{align}
By \eqref{e:KG} and \eqref{e:KG2},
\begin{align*}
I_3
\,\le\,& R^{-1}\int\limits_{E \cap \{  y: y_d \le \delta_D({y})   \}}
 |\wt y|^2 v(y_d)j(|x-y|) dy\\
 &+ R^{-1}\int\limits_{E \cap \{  y:y_d  > \delta_D({y})   \}}
|\wt y|^2  v\left(R-\sqrt{ |\wt y|^2+(R-y_d)^2}\right) j(|x-y|) dy\\
=:&R^{-1}( L_1+L_2).
\end{align*}

Since
\[
E\subset  \{z=(\wt z, z_d)\in \R^d: \ |\widetilde{z}|< \frac{R}{4}\wedge \sqrt{
2Rz_d-z_d^2}
\hbox{ and }  0< z_d\leq \frac{R}{2}\},
\]
changing to polar coordinates for $\wt y$ and using \eqref{e:Berall},
\eqref{e:Berall1},
 \eqref{prop:pot-alpha01} and 
Proposition \ref{e:behofV}, 
yields
\begin{eqnarray*}
L_1
&\leq & c_{6}\int\limits_{0}^{\frac{R}{2}}  v(y_d)\left(
\int\limits_0^{ \frac{R}{4}\wedge \sqrt{ 2Ry_d-y_d^2}} \frac{r^{d}  \phi'((r^2+
|y_d-
x_d|^2)^{-1}) }
{(r^2+ |y_d- x_d|^2)^{(d+2)/2}} dr
\right) dy_d\\
&\leq & c_{7}\int\limits_{0}^{\frac{R}{2}}  
v(y_d/6)
\left(
\int\limits_0^{ {R}} \frac{ r^{d} \phi'((r+
|y_d-
x_d|)^{-2}) }
{(r+ |y_d- x_d|)^{d+2}}
 dr
\right) dy_d=:c_7 \wh L_1
\end{eqnarray*}

If $\delta \not=\frac{1}{2}$, by {\bf(A-3)}
\begin{align}
& \int\limits_0^{ R} \frac{  \phi'((r+ |y_d- x_d|)^{-2}) }
{(r+ |y_d- x_d|)^2}
 dr \nn\\=& \phi'((R+ |y_d- x_d|)^{-2})  \int\limits_0^{ R}
\frac{\phi'((r+ |y_d- x_d|)^{-2}) }{\phi'((R+ |y_d-
x_d|)^{-2})}\frac{dr}{(r+ |y_d- x_d|)^2}
\nn \\
 \le& c_8  \phi'((R+ |y_d- x_d|)^{-2}) \int\limits_0^{ R}
\left(\frac{(r+ |y_d- x_d|)^{-2} }{(R+ |y_d-
x_d|)^{-2}}\right)^{-\delta}\frac{dr}{(r+ |y_d- x_d|)^2}
\nn \\
  =& c_8  \phi'((R+ |y_d- x_d|)^{-2})   (R+ |y_d-
x_d|)^{-2\delta} \int\limits_0^{ R} (r+ |y_d- x_d|)^{-2+2\delta}
 dr\nn\\
 \le& c_9  \phi'((R+ |y_d- x_d|)^{-2} )  (R+ |y_d-
x_d|)^{-2\delta} |y_d- x_d|^{-(1-2\delta)_+}  \label{e:hjhk1}
\end{align}

Thus, in the case  $\delta > \frac{1}{2}$, 
\eqref{e:hjhk1} implies
\begin{eqnarray}
\wh L_1
&\leq & c_{11}\int\limits_{0}^{\frac{R}{2}}  
v(y_d/6)  dy_d \le c_{12}
V(\tfrac{R}{12}) <
\infty. \label{new:234}
\end{eqnarray}

For  the case  $\delta \le \frac{1}{2}$, we first note that  using \eqref{e:Berall1} we obtain
\begin{eqnarray*}
&&\int\limits_0^{ {R}} \frac{ r^{d} \phi'((r+
|y_d-
x_d|)^{-2}) }
{(r+ |y_d- x_d|)^{d+2}}
 dr
 \\
 &\le& 
 \int\limits_0^{ |s-
x_d|} \frac{ r^{d} \phi((r+
|s-
x_d|)^{-2}) }
{(r+ |s- x_d|)^{d}}
 dr
+
\int\limits_{|s-
x_d|}^{ {R}} \frac{ r^{d} \phi((r+
|s-
x_d|)^{-2}) }
{(r+ |s- x_d|)^{d}}
 dr
\\
&\leq & \frac{ \phi(
|s-
x_d|^{-2}) }
{|s- x_d|^{d}}
\int\limits_0^{ |s-
x_d|} 
r^d dr
+
\int\limits_{|s-
x_d|}^{ {R}} \phi(r^{-2}) 
 dr
\\
&= & (d+1)^{-1}\phi(
|s-
x_d|^{-2}) 
|s- x_d| +
\int\limits_{|s-
x_d|}^{ {R}} \phi(r^{-2}) 
 dr
\end{eqnarray*}
Thus, by Lemma \ref{l:new}, 
\begin{align*}
\wh L_1  &\le c_{13}\int\limits_{0}^{\frac{R}{2}} 
 v(s/6)
\left(
\phi(
|s-
x_d|^{-2}) 
|s- x_d| +
\int\limits_{|s-
x_d|}^{ {R}} \phi(r^{-2}) 
 dr
\right) ds  < \infty.
\end{align*}

Let us estimate $L_2$. 
Switching to polar coordinates for $\wt y$,
and by the use of \eqref{prop:pot-alpha01},  
we get
\begin{align*}
&L_2\leq c_{20}
\int\limits_{0}^{x_d+\frac{R}{4}} \left(
\int\limits_0^{  \sqrt{ 2Ry_d-y_d^2}}
 v(R-\sqrt{ r^2+(R-y_d)^2})r^d
j((r^2+ |y_d- x_d|^2)^{1/2}) dr
\right) dy_d\\
&\leq  c_{21}\int\limits_{0}^{x_d+\frac{R}{4}}\left( \int\limits_0^{  \sqrt{
2Ry_d-y_d^2}}  \frac {   v(R-\sqrt{ r^2+(R-y_d)^2})
\phi'((r^2+ |y_d- x_d|^2)^{-1})}{ (r^2+ |y_d- x_d|^2)^{(d+2)/2}}
r^{d} dr
\right) dy_d\\
&\leq  c_{22}\int\limits_{0}^{x_d+\frac{R}{4}}  \left( \int\limits_0^{
\sqrt{
2Ry_d-y_d^2}} \frac {   v(R-\sqrt{ r^2+(R-y_d)^2})
\phi'((r+ |y_d- x_d|)^{-2}) }{ (r+ |y_d- x_d|)^{2+d}}
r^{d} dr
\right) dy_d  .
\end{align*}
Since, for $0<r<R$,
$$
{R-\sqrt{ r^2+(R-y_d)^2}}
=
\tfrac{ \left(\sqrt{2Ry_d-y_d^2}+r\right)  \left(\sqrt{2Ry_d-y_d^2}-r\right)    
    }{R+\sqrt{
r^2+(R-y_d)^2} }\,\ge\,
\tfrac{\sqrt{y_d}}{3 \sqrt R}  \left(\sqrt{2Ry_d-y_d^2}-r\right)
$$
and $\sqrt{
2Ry_d-y_d^2} <\sqrt{R/2}\sqrt{
2R-y_d}<R$ for $0<y_d <x_d+\frac{R}{4}$, 
we have
$$
L_2 \le c_{22}  \int\limits_{0}^{x_d+\frac{R}{4}}     \int\limits_0^{  \sqrt{
2Ry_d-y_d^2}}
\frac{v\left(\sqrt{y_d}  (\sqrt{2Ry_d-y_d^2}-r)/ (3 \sqrt R)\right)\phi'((r+
|y_d- x_d|)^{-2}) }{(r+ |y_d- x_d|)^{2+d}}r^d drdy_d\, .
$$

Using \eqref{e:Berall1},
we see that with $a:=  \sqrt{ 2Ry_d-y_d^2}  $ and $b:=  |y_d- x_d| $,
\begin{align*}
& \int\limits_0^{  a}  \frac{v(\sqrt{y_d}  (a-r)/(3 \sqrt R)) \phi'((r+
b)^{-2})}{(r+ b)^{2+d}}r^d dr\\
\le & \int\limits_0^{  a/2}  \frac{v(\sqrt{y_d}  (a-r)/(3 \sqrt R)) \phi'((r+
b)^{-2})}{(r+ b)^{2+d}}r^d dr+ 
\int\limits_{a/2}^{  a}  \frac{v(\sqrt{y_d}  (a-r)/(3 \sqrt R)) \phi((r+
b)^{-2})}{(r+ b)^{d}}r^d dr
\\
\le & v(\sqrt{y_d} a/(6 \sqrt R))  \int\limits_0^{ 
a/2}\frac{\phi'((r+
b)^{-2}) }{(r+ b)^{2+d}}r^d dr+\phi((
b+a/2)^{-2}) \int\limits_{
a/2}^{
a}
   v(\sqrt{y_d}  (a-r)/(3 \sqrt R))dr    \\
\le&v(\sqrt{y_d}  a/(6 \sqrt R))  \int\limits_0^{ R}\frac{\phi'((r+
b)^{-2})}{(r+ b)^{2+d}}r^ddr+c_{23}\phi((
b+a/2)^{-2})  \frac1{\sqrt{y_d}} 
V(\sqrt{y_d}  a/(6 \sqrt R))\\
:=&B_1(y_d)+ B_2(y_d) .\end{align*}

First, note that $\sqrt{y_dR}<   \sqrt{ 2Ry_d-y_d^2} =a \le \sqrt{ y_d}\sqrt{2 R}$. Thus 
\begin{align*}
&\int\limits_{0}^{\frac{R}{2}} B_1(y_d)dy_d \le 
c_{24}\int\limits_{0}^{\frac{R}{2}}
v(y_d/6) \int\limits_0^{R}\frac{\phi'((r+ |y_d- x_d|)^{-2})}{(r+ |y_d-
x_d|)^{2+d}} r^ddr 
dy_d =c_{24}\wh L_1 <\infty. \end{align*}

Using the inequality
 $y_d/\sqrt{R} \le \sqrt{ y_d}\le  |y_d- x_d|+ \sqrt{ y_d}$, we have
 $$\phi((|y_d- x_d|+\sqrt{ y_d})^{-2})
\le \phi((y_d/\sqrt{R})^{-2})^{1/2} \phi(y_d^{-1})^{1/2}.$$
This and  the inequality
$\sqrt{y_dR}< a \le \sqrt{ y_d}\sqrt{2 R}$, by  \eqref{e:Berall} and Proposition \ref{e:behofV}, then by Lemma \ref{p:USC} with $\eps=\delta/2$ we have 
\begin{align*}
B_2(y_d) 
&\le c_{25} y_d^{-1/2}\phi((
|y_d- x_d|+\sqrt{ y_d})^{-2})
\phi({y_d}^{-2})^{-1/2}\\
&\le  c_{25} y_d^{-1/2}\frac{\phi((
y_d/\sqrt{R})^{-2})^{1/2}}
{\phi({y_d}^{-2})^{1/2}}
\phi(y_d^{-1})^{1/2}
\\
&\le  c_{26} y_d^{-1/2}
\phi(y_d^{-1})^{1/2}  
\le c_{27} y_d^{-1/2}
y_d^{(-1+\delta-\eps)/2} = c_{27}  y_d^{-1+(\delta-\eps)/2}.
\end{align*}
Thus 
$$
\int\limits_{0}^{\frac{R}{2}} B_2(y_d)dy_d   \le 
c_{27}
   \int\limits_{0}^{\frac{R}{2}}   
y_d^{-1+(\delta-\eps)/2} dy_d < \infty.
$$
Therefore $L_2 < \infty.$

Now we see that  $\sA (h-h_x)(x)$ is well defined. Indeed, since $h_x(x)=h(x)$ and 
\begin{align*}
&{\bf 1}_{\{y \in D \cup H^+: \, |y-x|>\varepsilon\}}{|h(y)-h_{x}(y)|}j(|y-x|)
\\
&\le {\bf 1}_{A \cup B(x,\frac{R}{4})^c} (h(y)+h_{x}(y))
j(|y-x|)+{\bf 1}_{E}{|h(y)-h_{x}(y)|}j(|y-x|)  \in L^1(\R^d),\end{align*}
we can use the dominated convergence theorem to deduce that limit
\[
	\lim_{\varepsilon\downarrow 0}\int\limits_{\{y\in D\cup H^+\colon
|y-x|>\varepsilon\}}(h(y)-h_x(y))j(|y-x|)\,dy
\]
exists. Moreover, $\sA h(x)$ is then also well defined and satisfies $|\sA
h(x)|\leq C_1$ .
\qed

For $a, b>0$, we define
$D_Q( a, b) :=\left\{ y\in D: a >\rho_Q(y) >0,\, |\wt y | < b
\right\}$.

\begin{lemma}\label{L:2}
Assume additionally that  {\bf(A-5)} holds. 
There are constants
$R_1=R_1( R, \Lambda, \phi)\in (0, \frac{R}{16\sqrt{1+(1+ \Lambda)^2}})$ and
$c_i=c_i(R,  \Lambda, \phi)>0$,
$i=1, 2$,
such
that for every $r \le R_1$, $Q \in
\partial D$ and $x \in D_Q ( r, r)$,
 \bee\label{e:L:2}
\P_{x}\left(X_{ \tau_{ D_Q ( r, r)}} \in  D\right) \ge c_1 V( \delta_D (x))
 \eee
and
 \bee\label{e:L:3}
\E_x\left[ \tau_{ D_Q  ( r, r )}\right]\,\le\, 
c_2 V(\delta_D (x)).
 \eee
\end{lemma}

\pf
Without loss of generality we may assume that $Q=0$ and that $\psi\colon
\R^{d-1}\rightarrow \R$ is a $C^{1,1}$ function such that in the coordinate
system $CS_0$
\[
 B(0,R)\cap D=\{(\wt y,y_d)\in B(0,R) \text{ in }\ CS_0\colon y_d>\psi(\wt
y)\}\,.
\]

The function $\rho$ defined by  $\rho(y)=y_d-\psi(\wt y)$ 
satisfies
\begin{equation}\label{eq:r1}
 \frac{\rho(y)}{\sqrt{1+\Lambda^2}}\leq \delta_D(y)\leq \rho(y) \ \ \text{ for
all }\ \ y\in B(0,R)\cap D. 
\end{equation}

Define for $a>0$, 
\[
 D_a=\{y\in D\colon 0<\rho(y)<a,|\wt y|<a\}\,
\]
and 
the function 
\[
h(y)=
 \begin{cases}
  V(\delta_D(y)) & y\in 
  B(0, R)\cap D\\
0 & \text{otherwise}\,.
 \end{cases}
\]

Using Dynkin formula and 
the same approximation argument as in the proof of the Lemma 4.5 in \cite{KSV2}, 
from our Lemma \ref{L:Main} we
have the following estimate for any open set $U\subset B(0,
\tfrac{R}{4})\cap D$:
\begin{equation}\label{eq:rest1}
 h(x)-C_1\E_x\tau_U\leq \E_x h(X_{\tau_U})\leq h(x)+C_1\E_x\tau_U
\end{equation}
where $C_1>0$ is the constant from Lemma \ref{L:Main}. 

By choosing $A:=\frac{R}{4\sqrt{1+(1+\Lambda)^2}}$ we obtain
\[
 D_r\subset D_A\subset D(0,\tfrac{R}{4})\cap D\ \ \text{ for all } \ r\leq A\,.
\]

Indeed, for 
$y\in D_r$ and $r>0$ the following is true
\begin{equation}\label{eq:rest2}
 |y|^2=|\wt y|^2+|y_d|^2\leq r^2+(|y_d-\psi(\wt y)|+|\psi(\wt y)|)^2\leq
(1+(1+\Lambda)^2)r^2\,.
\end{equation}
In particular, for $r\leq A$
\[
  |y|\leq  \sqrt{1+(1+\Lambda)^2}A
 =\tfrac{R}{4}\,.
\]

The idea is to choose $
\lambda_2\geq 1$ large enough so that (\ref{e:L:2}) and
(\ref{e:L:3}) hold for $r\leq 
\lambda_2^{-1}A$ and $x\in D_r$\,.

We are going to show that there are constants $c_1,
c_2>0$ such that for any
$\lambda\geq 4$ and $x\in D_{\lambda^{-1}A}$ the following two inequalities
hold:
\begin{align}
 \E_x[h(X_{\tau_{D_{\lambda^{-1}A}}})]&\geq
c_1\left(\sqrt{\phi(16\lambda^2R^{-2})}-\sqrt{\phi(R^{-2})}\right)\E_x\tau_{D_{
\lambda^{-1}A}}\label{eq:ml1}\\
\P_x\left(X_{\tau_{D_{\lambda^{-1}A}}}\in D\right)&\geq
c_2\left(\phi(16\lambda^2R^{-2})-\phi(R^{-2})\right)\E_x\tau_{D_{
\lambda^{-1}A}}\label{eq:ml2}
\end{align}

Once we prove this, we can choose $
\lambda_2>4$ so that 
\[
 \sqrt{\phi(16
\lambda_2^2 R^{-2})}>\sqrt{\phi(R^{-2})}+\tfrac{2C_1}{c_1}\,.
\]

Then, for any $\lambda\geq 
\lambda_2$ and $x\in D_{\lambda^{-1}A}$ we can use 
\[
 c_1\left(\sqrt{\phi(16\lambda^2R^{-2})}-\sqrt{\phi(R^{-2})}\right)-C_1>C_1
\]
on \eqref{eq:ml1} and (\ref{eq:rest1}) to get
\[
 V(\delta_D(x))=h(x)\geq \E_x[h(X_{\tau_{D_{\lambda^{-1}A}}})]-C_1\E_x\tau_{D_{
\lambda^{-1}A}}\geq C_1 \E_x\tau_{D_{
\lambda^{-1}A}},
\]
which proves (\ref{e:L:3}) with $R_1=
\lambda_2^{-1}A$ .

Similarly, 
by (\ref{eq:rest1}) and (\ref{eq:ml2}), for any $\lambda\geq 
\lambda_2$ and $x\in D_{\lambda^{-1}A}$ we have 
\begin{align*}
 &V(\delta_D(x))\,=
 \,h(x)\,\leq\, \E_x[h(X_{\tau_{D_{\lambda^{-1}A}}})]+C_1\E_x\tau_{D_{
\lambda^{-1}A}}\\
&\leq V(R)\P_x\left(X_{\tau_{D_{\lambda^{-1}A}}}\in D\right)+
C_1c_2^{-1} \left(\phi(16
\lambda_2^2R^{-2})-\phi(R^{-2})\right)^{-1}\P_x\left(X_{\tau_{D_{\lambda^{-1}A}}}\in D\right),
\end{align*}
where the first term is obtained by estimating $h$ by $V (R)$ and noting that $h(x) = 0$ unless $x \in D$.
This  yields
\[
 \P_x\left(X_{\tau_{D_{\lambda^{-1}A}}}\in D\right)\geq
\frac{V(\delta_D(x))}{V(R)+
C_1c_2^{-1}\left(\phi(16
\lambda_2^2R^{-2})-\phi(R^{-2})\right)^{-1} }\,.
\]
This proves (\ref{e:L:2}) with $R_1=
\lambda_2^{-1}A$ .

Now we prove (\ref{eq:ml1}). Note that for $z\in D_{\lambda^{-1}A}$ and
$y\not\in B(0,\lambda^{-1}\tfrac{R}{r})$, 
\begin{align}\label{newineq}
 |z|\leq \sqrt{1+(1+\lambda^2)}\lambda^{-1}A=\lambda^{-1}\tfrac{R}{4}\leq |y|
\end{align}
implies
\[
 j(|z-y|)\geq j(2|y|)\geq c_3j(|y|)\,.
\]
Then the Ikeda-Watanabe formula implies
\begin{align*}
 \E_x[h(X_{\tau_{D_{\lambda^{-1}A}}})]&\geq\int\limits_{B(0,r)\cap D\setminus
D_{\lambda^{-1}A}}\int\limits_{D_{\lambda^{-1}A}}G_{D_{\lambda^{-1}A}}(x,
z)j(|z-y|)V(\delta_D(y))\,dz\,dy\\
&\geq c_3 \left(\int\limits_{D_{\lambda^{-1}A}}G_{D_{\lambda^{-1}A}}(x,
z)\,dz\right)\int\limits_{B(0,R)\cap D\setminus
D_{\lambda^{-1}A}}V(\delta_D(y))j(|y|)\,dy\\
&\geq c_3\E_x\tau_{D_{\lambda^{-1}A}}\int\limits_{B(0,R)\cap D\setminus
D_{\lambda^{-1}A}}j(|y|)V\left(\tfrac{y_d-\psi(\wt
y)}{\sqrt{1+
\Lambda^2}}\right)\,dy,
\end{align*}
since $\frac{y_d-\psi(\wt
y)}{\sqrt{1+
\Lambda^2}}\leq \delta_D(y)$ by (\ref{eq:r1})\,.

On the set $E:=\{(\wt y,y_d)\colon 2\Lambda|\wt y|<y_d,\,
\lambda^{-1}\tfrac{R}{4}<|y|<R\}$ we have
\[
 |y|\leq \sqrt{1+4\Lambda^2}\,y_d\ \ \text{ and }\ \ y_d-\psi(\wt y)\geq
y_d-\Lambda|\wt y|\geq \tfrac{|y|}{2\sqrt{1+4\Lambda^2}}\,.
\]
Since $E\subset B(0,R)\setminus D_{\lambda^{-1}A}$
because of the first inequality in \eqref{newineq}, 
changing to polar coordinates
gives
\[
 \E_x[h(X_{\tau_{D_{\lambda^{-1}A}}})]\geq
c_4\E_x[\tau_{D_{\lambda^{-1}A}}]\int\limits_{\lambda^{-1}\frac{R}{4}}
^Rj(r)V(\tfrac{r} 
{2\sqrt{1+4\Lambda^2}\sqrt{1+\Lambda^2}}
)r^{d-1}\,dr
\]
with constant $c_4>0$ depending on $\Lambda$ and $d$. 

Then \eqref{prop:pot-alpha01} and Proposition \ref{e:behofV} imply
\begin{align*}
 \E_x[h(X_{\tau_{D_{\lambda^{-1}A}}})]&\geq
c_5\E_x[\tau_{D_{\lambda^{-1}A}}]\int\limits_{\lambda^{-1}\frac{R}{4}}^Rr^{-3}
\tfrac{
\phi'(r^{-2})} {\sqrt{\phi(r^{-2})}}\,dr\\
&=c_5\E_x[\tau_{D_{\lambda^{-1}A}}]\left(\sqrt{\phi(16\lambda^2R^{-2})}-\sqrt{
\phi(R^{-2})}\right)\,.
\end{align*}

We prove (\ref{eq:ml2}) similarly by the same computation as above without $V$:
\begin{align*}
\P_x\left(X_{\tau_{D_{\lambda^{-1}A}}}\in
D\right)&\geq\P_x\left(X_{\tau_{D_{\lambda^{-1}A}}}\in B(0,R)\cap
D\setminus B(0,\lambda^{-1}\tfrac{R}{4})\right)\\
&\geq c_6\E_x[\tau_{D_{\lambda^{-1}A}}]\int\limits_{\lambda^{-1}\frac{R}{4}}
^Rj(r) r^{d-1}\,dr\\
&\geq c_7\E_x[\tau_{D_{\lambda^{-1}A}}]\int\limits_{\lambda^{-1}\frac{R}{4}}^R
r^{-3}\phi'(r^{-2})\,dr\\
&=2^{-1}c_7\E_x[\tau_{D_{\lambda^{-1}A}}]\left(\phi(16\lambda^2R^{-2})-\phi(R^{-2}
)\right)\,.
\end{align*}
\qed

\section{Analysis of Poisson Kernel}\label{sec:pk}

In this section we always assume that 
 the Laplace exponent $\phi$ of the  subordinator $S=(S_t\colon t\geq 0)$
satisfies {\bf(A-1)}--{\bf(A-4)} 
and the corresponding subordinate Brownian motion $X=(X_t,\P_x)$ is  transient.

First we record an inequality.  
\begin{lemma}\label{l:l}
For every $R_0>0$, there exists a constant 
$c(R_0, \phi)>0$ such that
 \begin{equation}\label{e:fsagd}
\lambda^2 \int\limits_0^{\lambda^{-1}} r^{-1} \phi' (r^{-2})dr +
 \int\limits_{\lambda^{-1}}^{R_0}r^{-3} \phi' (r^{-2})dr \,\le\, c\,
\phi(\lambda^2), \quad \forall \lambda \ge \tfrac{1}{R_0}.
\end{equation}
\end{lemma}

\proof 
Assume $ \lambda \ge \lambda_0 \vee \tfrac{1}{R_0}$. 
By \eqref{eq:as-1}, $ \phi'(r^{-2})
\le
c_1 r^{2\delta}\lambda^{2\delta}  \phi'(\lambda^{2}) $ for $r \le \lambda^{-1}$.
Thus
\begin{align*}
&\lambda^2 \int\limits_0^{\lambda^{-1}} r^{-1} \phi' (r^{-2})dr +
\int\limits_{\lambda^{-1}}^{R_0} r^{-3} \phi' (r^{-2})dr\\
&=\lambda^2 \phi' (\lambda^2)\int\limits_0^{\lambda^{-1}} r^{-1} \frac{\phi'
(r^{-2})}{\phi' (\lambda^2)} dr - 
\frac{1}{2}
\int\limits_{\lambda^{-1}}^{R_0}  (\phi (r^{-2}))'dr\\
&\le  c_2 \phi' (\lambda^2)\lambda^{2+2\delta} 
\int\limits_0^{\lambda^{-1}} r^{-1+2\delta}   dr +c_2\phi(\lambda^{2}) \,
\le\,  c_3 (\phi' (\lambda^2)\lambda^{2}+\phi(\lambda^{2}) ) \le 2 c_3
\phi(\lambda^{2})
\end{align*}
where we have used \eqref{e:Berall1} in the last inequality. 

If $\tfrac{1}{R_0} > \lambda_0$ and $\tfrac{1}{R_0} \le \lambda \le  \lambda_0$, then clearly 
the left hand side of \eqref{e:fsagd} is bounded above by 
\begin{align*}
&\lambda_0^2 \int\limits_0^{R_0} r^{-1} \phi' (r^{-2})dr +
 \int\limits_{\lambda_0^{-1}}^{R_0}r^{-3} \phi' (r^{-2})dr
= c_4
 \le c_5 \phi(\lambda^2).
\end{align*}
\qed

Recall that the infinitesimal generator $\sL$ of $X$
is given by
\begin{equation}\label{3.1}
\sL f(x)=\int\limits_{\R^d}\left( f(x+y)-f(x)-y\cdot \nabla f(x)
{\bf 1}_{\{|y|\le
\eps \}}
 \right)\, j(|y|)dy
\end{equation}
for every $\eps>0$ and $f\in C_b^2(\R^d)$ where 
%$C^2_b$
$C^2_b(\R^d)$ is the collection of bounded $C^2$
functions in
$\RR^d$.

Using Lemma \ref{l:l}, we now prove \cite[Lemma 4.2]{KSV4} under a weaker
assumption. 

\begin{lemma}\label{l:lnew}
There exists a constant 
$c=c(\phi)>0$
 such that for every $f\in
C^2_b(\R^d)$ with $0\leq f \leq 1$,
$$
|\sL f_r(x)| \le   c \phi(r^{-2}) \left( 
1+ \sup_{y}\sum_{j,k}
\left|\tfrac{\partial^2f}{\partial
y_j\partial y_k} (y)\right| \right) + b_0, \quad \text{for every } x \in \R^d\
\text{ and }\ 
r \in (0,1],
$$
where $f_r(y):=f(\frac{y}{r})$ and $b_0:=2\int\limits_{|z| > 1} j(|z|)dz <
\infty$.
\end{lemma}
\proof
Set   $L_1=\sup_{y\in \R^d}\sum_{j,k} |\frac{\partial^2f(y)}{\partial
y_j\partial y_k}|$. Then \[|f(z+y)-f(z) -y\cdot \nabla f(z)|\le
\tfrac{1}{2}L_1 |y|^2 ,\] 
which implies the following estimate
\begin{eqnarray*}
|f_r(z+y)-f_r(z) -y\cdot \nabla f_r(z){\bf 1}_{\{|y|\le r\}}| \le
\tfrac{L_1}{2}
\tfrac{|y|^2}{r^2}{\bf 1}_{\{|y|\le r\}} + 2 \cdot {\bf 1}_{\{|y|\ge r\}}\ .
\end{eqnarray*}
Now,
 \eqref{prop:pot-alpha01} and \eqref{e:fsagd} yield
\begin{eqnarray*}
&&|\sL f_r(z)|\\
 &\le & \int\limits_{\R^d}
|f_r(z+y)-f_r(z) -y\cdot \nabla f_r(z){\bf 1}_{\{|y|\le r\}}| \, j(|y|)dy \\
& \le & \tfrac{L_1}2\int\limits_{\R^d} {\bf 1}_{\{|y|\le r\}}
\tfrac{|y|^2}{r^2} j(|y|)dy+2\int\limits_{\R^d}{\bf 1}_{\{r \le |y|\le 1\}} 
j(|y|)dy
+2\int\limits_{\R^d}{\bf 1}_{\{|y|\ge 1\}} j(|y|)dy\\
& \le &   c \phi(r^{-2}) \left( 2+\tfrac{L_1}2  \right) 
+ 2\int\limits_{\{|y|\ge 1\}}
j(|y|)dy \, ,
\end{eqnarray*}
where the constant $c$ is independent of $r\in (0,1]$.
\qed

\begin{lemma}\label{l2.1}
For every $a \in (0, 1)$, there exists a positive constant 
$c=c(a,\phi)>0$
such that
for any $r\in (0, 1)$ and any open set $D$ with $D\subset B(0, r)$ 
$$
{\P}_x\left(X_{\tau_D} \in B(0, r)^c\right) \,\le\, c\,
\phi(r^{-2})\,
\E_x[\tau_D] \ \ \text{ for all }\ \  x \in D\cap B(0,
ar)\, .
$$
\end{lemma}
\proof Using Lemma \ref{l:lnew}, 
the proof of the lemma is similar to that of \cite[Lemma 13.4.15]{KSV3}.
 We omit the details.
\qed

Let 
$A(x,a,b):=\{y\in \R^d\colon
a\leq |y-x|<b\}$
and recall that  the Poisson kernel $K_D(x,z)$ of $X$ in $D$ is defined in \eqref{eq:sub-10}.

Unlike \cite{KSV4}, instead of Harnack inequality we 
use
Corollary \ref{rem:green}
(wihch is a combination of Proposition \ref{p:green} and Harnack inequality) in the next proposition.

\begin{prop}\label{l:k_B}
	Let $p\in (0,1)$. Then there exists a constant 
	$c(\phi, p)>0$ such that for any
$r\in (0,1)$ we have
	\[
		\int\limits_{\frac{1+p}{2}\,r}^{|z|}K_{B(0,s)}(x,z)\,ds\leq
c\tfrac{r}{\phi(r^{-2})}\,j(|z|)
	\]
	for all $x\in B(0,pr)$ and $z\in A(0,\frac{1+p}{2}\,r,r)$.
\end{prop}
\proof
	We split the Poisson kernel into two parts:
	\[	
K_{B(0,s)}(x,z)=\int\limits_{B(0,s)}G_{B(0,s)}(x,y)j(|z-y|)\,dy=I_1(s)+I_2(s)
	\]
	where
	\begin{align*}
		I_1(s)&=\int\limits_{B(0,3s/4)}G_{B(0,s)}(x,y)j(|z-y|)\,dy\\
		I_2(s)&=\int\limits_{A(0,3s/4,s)}G_{B(0,s)}(x,y)j(|z-y|)\,dy.
	\end{align*}
	
	First we consider $I_1(s)$. 
	Since $|z-y|\geq \frac{1}{4}|z|$, 
we conclude from \eqref{e:Berall} and \eqref{prop:pot-alpha02} that
	\begin{align*}
		I_1(s)&\leq
j\left(\tfrac{|z|}{4}\right)\int\limits_{B(0,3s/4)}G(x,y)\,dy
\leq
j\left(\tfrac{|z|}{4}\right)\int\limits_{B(x,2s)}G(x,y)\,dy\\
		&\leq c_1  j\left(|z|\right)\int\limits_0^{2s}
		t^{-3}\tfrac{\phi'(t^{-2})}{\phi(t^{-2})^2}dt=\frac{c_1}2  j\left(|z|\right)\int\limits_0^{2s}
		\left(\tfrac{1}{\phi(t^{-2})}\right)'dt
		\leq c_2 \tfrac{j(|z|)}{\phi(s^{-2})}\,.
	\end{align*}
	Then,
since $|z|\leq r$, 
	\begin{align*}
		\int\limits_{\frac{1+p}{2}\,r}^{|z|}I_1(s)\,ds&\leq
c_2j(|z|)\int\limits_{\frac{1+p}{2}\,r}^{|z|} \tfrac{ds}{\phi(s^{-2})}\\
		&\leq c_2j(|z|)\tfrac{|z|-\frac{1+p}{2}r}{\phi(r^{-2})}\leq
c_2j(|z|)\tfrac{r}{\phi(r^{-2})}\,.
	\end{align*}
	
 On the other hand, by 
 %Remark \ref{rem:green}
 Corollary \ref{rem:green}
 and Lemma \ref{l:tau}, 
	\begin{align*}
		I_2(s)&\leq
c_3s^{-d-2}\tfrac{\phi'(s^{-2})}{\phi(s^{-2})}\int\limits_{A(0,3s/4,s)}\E_y[
\tau_{B(0, s)}]\,j(|z-y|)\,dy\\
		&\leq
c_4s^{-d-2}\tfrac{\phi'(s^{-2})}{\phi(s^{-2})}\int\limits_{A(0,3s/4,s)}\tfrac{
j(|z-y|)}{\sqrt{\phi(s^{-2})\phi((s-|y|)^{-2})}}\,dy\\
		&\leq
c_4s^{-d-2}\tfrac{\phi'(s^{-2})}{\phi(s^{-2})^{3/2}}\int\limits_{A(0,3s/4,s)}
\tfrac{j(|z-y|)}{\sqrt{\phi(|z-y|^{-2})}}\,dy,
	\end{align*}
	since $s-|y|\leq |z-y|$.
	
	Observing that  $A(z,3s/4,s)\subset B(z,s)\subset 
	 A(0,|z|-s,2r) $ we arrive at
	\begin{align*}
		I_2(s)&\leq
c_4s^{-d-2}\tfrac{\phi'(s^{-2})}{\phi(s^{-2})^{3/2}}\int\limits_{A(0,|z|-s,2r)}
\tfrac{
j(|v|)}{\sqrt{\phi(|v|^{-2})}}\,dv\\
			&=c_5
s^{-d-2}\tfrac{\phi'(s^{-2})}{\phi(s^{-2})^{3/2}}\int\limits_{|z|-s}^{2r}t^{-3}
\tfrac{
\phi'(t^{-2})}{\sqrt{\phi(t^{-2})}}\,dv\\
			&\leq c_6
s^{-d-2}\tfrac{\phi'(s^{-2})}{\phi(s^{-2})^{3/2}}\sqrt{\phi((|z|-s))^{-2}}.
	\end{align*}
	
Then using the fact that $s\mapsto \phi'(s^{-2})$ and $s\mapsto
\phi(s^{-2})^{-1}$ are increasing we obtain
\begin{align}
		\int\limits_{\frac{1+p}{2}\,r}^{|z|}I_2(s)\,ds&\leq c_6
\int\limits_{\frac{1+p}{2}\,r}^{|z|}\tfrac{s^{-d-2}\phi'(s^{-2})}{\phi(s^{-2})^{
3/2}}
\sqrt{\phi((|z|-s)^{-2})}\,ds\nonumber\\
		&\leq c_6
\tfrac{\left(\frac{1+p}{2}r\right)^{-d-2}\phi'(|z|^{-2})}{\phi(|z|^{-2})^{3/2}}
\int\limits_0^{|z|-\frac{1+p}{2}r}\sqrt{\phi(t^{-2})}\,dt\,.\label{eq:tmp-p_est}
	\end{align}

By Lemma \ref{p:USC} with 	
$\eps=\frac{\delta}{2}>0$ for any $a\in (0,1)$ we have
\begin{align}
	\int\limits_0^a \sqrt{\phi(s^{-2})}\,ds &= \int\limits_0^a
\tfrac{\sqrt{\phi(s^{-2})}}{\sqrt{\phi(a^{-2})}}\,ds{\sqrt{\phi(a^{-2})}}\nn\\
&	\le c_7 a^{1-\delta/2} {\sqrt{\phi(a^{-2})}} \int\limits_0^a
s^{-1+\delta/2}ds \le 
	 c_8 a \sqrt{\phi(a^{-2})} \label{e:newp1}
\end{align}
Since $\frac{1+p}{2}r \le  |z|  \le r$, 
 (\ref{eq:tmp-p_est})--\eqref{e:newp1}  together with  (\ref{eq:decr}) and \eqref{prop:pot-alpha01} give 
\begin{align*}
	\int\limits_{\frac{1+p}{2}\,r}^{|z|}I_2(s)\,ds&\leq c_9
\tfrac{\left(\tfrac{1+p}{2}r\right)^{-d-2}\phi'(|z|^{-2})}{\phi(|z|^{-2})^{3/2}}
\left(|z|-\tfrac{1+p}{2}r\right)\phi\left(\left(|z|-\tfrac{1+p}{2}r\right)^
{ -2}\right)^{1/2}\\
&\leq c_9
\tfrac{\left(\frac{1+p}{2}r\right)^{-d-2}\phi'(|z|^{-2})}{\phi(|z|^{-2})^{3/2}}
|z|\sqrt{\phi\left(|z|^{-2}\right)}\leq 
c_9
|z|^{-d-2}\phi'(|z|^{-2})
\tfrac{r}{\phi(r^{-2})}\\
&\leq c_{10} j(|z|)\tfrac{r}{\phi(r^{-2})}.
\end{align*}		
\qed

\section{Uniform Boundary Harnack Principle}\label{sec:bhp}

In this section we give a proof of the uniform boundary Harnack
principle
for
$X$ in an arbitrary open set with the constant not depending on the open set
itself. 
This type of the  boundary Harnack principle was first obtained in
\cite{BKK} for rotationally symmetric stable processes. 
Since, using results of previous section, the proofs in this section are almost
identical to the one in \cite[Section 5]{KSV4}, we give details only on parts
that require extra explanation.

Recall that $X=(X_t,\P_x)$ is a subordinate process
 defined by $X_t=W_{S_t}$ where $W=(W_t,\P_x)$ is  a Brownian motion in $\R^d$
independent of the  subordinator $S$ and 
the Laplace exponent $\phi$ of the  subordinator $S$
satisfies {\bf(A-1)}--{\bf(A-3)}.

Using  \eqref{H:1}, \eqref{eq:sub-11},
Proposition \ref{p:Poisson1}, 
Proposition \ref{l:k_B} and
the fact that 
for $U \subset D$
\beq\label{e:KDU}
K_D(x,z)= K_U(x,z) + \E_x\left[   K_D(X_{\tau_U}, z)\right], \qquad (x,z) \in U
\times D^c,
\eeq
the proof of the next result is the same as the one of \cite[Lemma 5.2]{KSV4}.

\begin{lemma}\label{lK1_1} Assume that $X$ is transient and satisfies
{\bf(A-1)}--{\bf(A-4)}. 
For every $p \in (0,1)$, there exists $c=
c( \phi, p)>0$ such that for every $r \in (0, 1)$, $z_0 \in \R^d$, $U \subset
B(z_0,r)$ and for any  $ (x,y) \in (U\cap B(z_0, pr)) \times B(z_0, r)^c$,
\begin{eqnarray*}
K_U(x, y)
\le
c\,\tfrac{1}{\phi(r^{-2})} 
\left(\int\limits_{U\setminus B\left(z_0, \frac{(1+p)r}{2}\right)}  j(|z-z_0|) 
K_U(z,y)dz +j(|y-z_0|) \right).
\end{eqnarray*}
\end{lemma}

The process $X$ satisfies the hypothesis ${\bf H}$ in \cite{Sz}. Therefore, 
by \cite[Theorem 1]{Sz},  for  a Lipschitz open set $V\subset \R^d$ and an open
subset $U \subset V$
\beq\label{e:Lip}
\P_x(X_{\tau_U} \in \partial V)=0 \qquad \text{ and } \qquad \P_x(X_{\tau_U} \in
dz) = K_U(x,z)dz \quad\text{on }V^c.
\eeq
Using \eqref{e:Lip} and 
Lemma \ref{lK1_1}, 
the proof of the next result is the same as the one of \cite[Lemma 5.3]{KSV4}.

\begin{lemma}\label{lK1}
Assume that $X$ is transient and satisfies
{\bf(A-1)}--{\bf(A-4)}. 
For every $p \in (0,1)$, there exists $c=
c( \phi, p)>0$ such that for every
$r \in (0, 1)$, for every $z_0 \in \R^d$, $U \subset B(z_0,r)$ and  any
nonnegative function $u$ in $\RR^d$
which is regular harmonic in $U$ with respect to $X$ and vanishes in
$U^c \cap B(z_0, r)$ we have
\begin{eqnarray*}
u(x)
\le
c\,\tfrac{1}{\phi(r^{-2})}
\int\limits_{\left(U\setminus B(z_0, \frac{(1+p)r}{2})\right) \cup B(z_0,r)^c} 
j(
|y-z_0|) u(y)dy,\quad
x \in U \cap B(z_0, pr).
\end{eqnarray*}
\end{lemma}

We give a detailed  proof of the next result.

\begin{lemma}\label{lK2_1}
Assume that $X$ is transient and satisfies
{\bf(A-1)}--{\bf(A-4)}. 
There exists $C_2=C_2( d,  \phi)>1$ such that for every $r \in (0, 1)$,
for every $z_0 \in \R^d$, $U \subset B(z_0,r)$ and for any  $ (x,y) \in U \cap
B(z_0, \frac{r}{2}) \times B(z_0, r)^c$,
\begin{eqnarray*}
\lefteqn{C^{-1}_{2}\, \E_x[\tau_{U}]
 \left(\int\limits_{U\setminus B(z_0, \frac{r}{2})}  j(|z-z_0|)  K_U(z,y)dz 
+j(|y-z_0|) \right)} \\
&\le &K_U(x,y) \le
C_2\,  \E_x[\tau_{U}] \left(\int\limits_{U\setminus B(z_0, \frac{r}{2})} 
j(|z-z_0|) K_U(z,y)dz  +j(|y-z_0|) \right).
 \end{eqnarray*}
 \end{lemma}
\proof
Without loss of generality, we assume $z_0=0$. Fix $r \in (0, 1)$ and let 
$U_1:=U \cap B(0, \frac12 r)$, $U_2:=U \cap B(0, \frac23 r)$
 and $U_3:=U \cap B(0, \frac34 r)$.
Let $x\in U\cap B(0, \frac{r}{2})$, $y\in B(0,r)^c$.
By \eqref{e:KDU},
\begin{eqnarray*}
K_{U}(x,y)&=&\E_x[K_U(X_{\tau_{U_2}},y)]  + K_{U_2}(x, y)\\
   &=&\int\limits_{U\setminus U_2}K_U(z,y)\P_x(X_{\tau_{U_2}} \in dz) +
K_{U_2}(x, y)\\
  &=&\int\limits_{U_3\setminus U_2}K_U(z,y)\P_x(X_{\tau_{U_2}} \in dz)
+\int\limits_{U\setminus U_3}K_U(z,y)K_{U_2}(x,z)dz+  K_{U_2}(x, y) \\
  &=&\int\limits_{U_3\setminus U_2}K_U(z,y)\P_x(X_{\tau_{U_2}} \in
dz)+\int\limits_{U\setminus U_3}K_U(z,y)\int\limits_{U_2}G_{U_2}(x,w)j(|z-w|)dw
dz\\
  &&\quad+
  \int\limits_{U_2}G_{U_2}(x,w)j(|y-w|)dw=:I_1+I_2+I_3.\\
 \end{eqnarray*}
From Lemma  \ref{l2.1} and Lemma \ref{lK1_1}, we see that there exist
$c_1$ and $c_2$ such that 
  \beq \label{e:k1}
 I_1\leq c_1\left(\sup_{z \in U_3}K_U(z,y)\right)\phi(r^{-2})\E_x[\tau_{U_2}]  
\,\le\, c_2\E_x[\tau_{U_2}]
 \left(\int\limits_{U\setminus U_3}   j( |z|)  K_U(z,y) dz  + j(|y|) \right).
\eeq

Now using \eqref{H:1} and \eqref{eq:sub-11} one can check as in \cite{KSV4} that
there exists $c_5=c_5(d, \phi)>1$ such that
  \beq \label{e:k2}
c_5^{-1}  \E_x[\tau_{U_2}]
  \int\limits_{U\setminus U_3}  j(|z|)  K_{U}(z,y)dz \le I_2 \le
c_5\E_x[\tau_{U_2}]\int\limits_{U\setminus U_3 } j(|z|)  K_{U}(z,y)dz
\eeq
and
 \beq \label{e:k2_1}
 c_5^{-1} \E_x[\tau_{U_2}]j(|y|) \le I_3  \le c_5 \E_x[\tau_{U_2}]j(|y|)\, .
\eeq
The upper bound follows from \eqref{e:k1}--\eqref{e:k2_1}.

Using the strong Markov property,  we get
  \begin{eqnarray*}
\E_x[\tau_{U}]&= &\E_x[\tau_{U_2}] + \E_x\left[  
\E_{X_{\tau_{U_2}}}[\tau_{U}]\right]\\
&\le &\E_x[\tau_{U_2}] + \left(\sup_{z \in U}\E_z[\tau_{U}]\right) 
\P_x\big(X_{\tau_{U_2}} \in B(0, \tfrac{2r}{3} )^c\big) \\
&\le &\E_x[\tau_{U_2}] + c_6\phi(r^{-2})^{-1}\,
\phi((\tfrac{2r}{3})^{-2})\E_x[\tau_{U_2}]\,\le \, c_7 \E_x[\tau_{U_2}],
 \end{eqnarray*}
where in the second inequality we have used  Lemma \ref{l:tau}
and Lemma \ref{l2.1} and in last inequality we have used \eqref{e:Berall}.

 Since
\begin{eqnarray*}
 \int\limits_{U\setminus U_1}  j(|z|)  K_U(z,y)dz
  &=&\int\limits_{U\setminus U_3}  j(|z|)  K_U(z,y)dz+\int\limits_{U_3 \setminus
U_1}
  j(|z|)  K_U(z,y)dz \\
    &\le &\int\limits_{U\setminus U_3}  j(|z|)  K_U(z,y)dz +\left(\sup_{z \in
U_3} K_U(z,y)
    \right) \int\limits_{A(0, r/2, 3r/4)}
   j(|y|)  dy,
    \end{eqnarray*}
  by \eqref{prop:pot-alpha01} and Lemma
  \ref{lK1_1},
\begin{eqnarray}
 &&\int\limits_{U\setminus U_1}  j(|z|)  K_U(z,y)dz  \nn\\
  & \le& \left(1+ \frac{c_8}{\phi(r^{-2})}\int\limits_{\frac{r}{2}}^{\frac{3r}{4}}  s^{-3}
\phi'(s^{-2})  ds \right)
 \left(\int\limits_{U\setminus U_3}  j(|z|)  K_U(z,y)dz +
j(|y|)\right)\nonumber\\
 & =& \left(1-2 \frac{c_8}{\phi(r^{-2})}\int\limits_{\frac{r}{2}}^{\frac{3r}{4}} 
(\phi(s^{-2}))'  ds \right)
 \left(\int\limits_{U\setminus U_3}  j(|z|)  K_U(z,y)dz +
j(|y|)\right)\nonumber\\
  &\le& 
  \left(1+c_{9} \frac{\phi( 4 r^{-2})}{\phi(r^{-2})}\right)
   \left(\int\limits_{U\setminus U_3}  j(|z|)  K_U(z,y)dz + j(|y|)\right).
\label{e:k4}
  \end{eqnarray}
Combining \eqref{e:Berall} and \eqref{e:k2}--\eqref{e:k4}, we finish the proof
of the lower bound.
\qed

Using Lemmas \ref{lK1} and \ref{lK2_1}, 
the proof of the next result is the same as the one of \cite[Lemma 5.5]{KSV4}.

\begin{lemma}\label{lK2}
Assume that $X$ is transient and satisfies
{\bf(A-1)}--{\bf(A-4)}. 
For every $z_0 \in \R^d$, every open set $U \subset B(z_0,r)$ and for any
nonnegative function $u$ in $\RR^d$ which is regular harmonic in $U$ with
respect to $X$ and vanishes a.e.~on $U^c \cap B(z_0, r)$
\begin{eqnarray*}
C_2^{-1}\E_x[\tau_{U}] \int\limits_{B(z_0, \frac{r}{2})^c}  j(|y-z_0|)  u(y)dy
\le u(x)\le
C_2 \E_x[\tau_{U}] \int\limits_{B(z_0, \frac{r}{2})^c} j(|y-z_0|)  u(y)dy
\end{eqnarray*}
for every  $ x \in U \cap B(z_0, \frac{r}{2})$ (where $
C_2$ is the constant from Lemma \ref{lK2_1}).
\end{lemma}

As in \cite[Corollary 5.6]{KSV4}, 
the last two lemmas immediately imply the following approximate factorization of
the Poisson kernel.
\begin{corollary}\label{c:approx-factor}
Assume that $X$ is transient and satisfies
{\bf(A-1)}--{\bf(A-4)}. 
Let $z_0\in \R^d$ and $D\subset \R^d$ be open. 
Then for every $r\in (0,1)$ and all $(x,y)\in (D\cap  B(z_0,\frac{r}{2}))\times
(D^c\cap B(z_0,r)^c)$  it holds that
\begin{equation}\label{e:approx-factor}
C_2^{-1} \E_x[\tau_
{D\cap B(z_0,r)}] 
A_D(z_0,r,y) \le  K_D(x,y) \le
C_2\, \E_x[\tau_
{D\cap B(z_0,r)}] 
A_D(z_0,r,y)\, ,
\end{equation}
where
\begin{eqnarray*}
A_D(z_0,r,y)&:=&\int\limits_{
({D\cap B(z_0,r)})\setminus B(z_0,\frac{r}{2})}j(|z-z_0|)K_
{D\cap B(z_0,r)}(z,y)\, dz\\
&&+j(|y-z_0|)+\int\limits_{B(z_0,\frac{r}{2})^c}
j(|z-z_0|)\E_z\left[K_D(X_{\tau_
{D\cap B(z_0,r)}},y)\right]\, dz \, .
\end{eqnarray*}
\end{corollary}

Lemma \ref{lK2} and \eqref{e:approx-factor} imply 
the following uniform boundary Harnack
principle. Note that the constants in the 
following  theorem does not depend on the open set itself. That is why this type
of result is called
 the uniform boundary Harnack principle.

\begin{thm}\label{UBHP} Suppose that $\phi$
satisfies {\bf(A-1)}--{\bf(A-3)}.
There exists a constant $c= c(\phi)>0$ such that
\begin{itemize}
    \item[(i)] For every $z_0 \in \R^d$, every open set $D\subset \R^d$, every
$r\in (0,1)$
    and for any nonnegative functions $u, v$ in $\R^d$ which are regular
harmonic in $D\cap
    B(z_0, r)$ with respect to $X$ and vanish 
    a.e.~on $D^c \cap B(z_0, r)$, we have
        $$
        \frac{u(x)}{v(x)}\,\le c\,\frac{u(y)}{v(y)}
        $$
        for all $x, y\in D\cap B(z_0, \frac{r}{2})$.
    \item[(ii)] If $X$ is, additionally, transient and satisfies {\bf(A-4)}, 
then for every $z_0 \in \R^d$, every
  Greenian open set
     $D\subset \R^d$, every $r\in (0,1)$, we have
    $$
    K_D(x_1, y_1)K_D(x_2, y_2) \le c K_D(x_1, y_2)K_D(x_2, y_1)
    $$
    for all $x_1, x_2 \in D \cap B(z_0, \frac{r}{2})$ and all $y_1, y_2 \in
    \overline{D}^c \cap B(z_0, r)^c$.
\end{itemize}
\end{thm}

\begin{proof}
Under the assumption of transience and {\bf(A-1)}--{\bf(A-4)} the result
follows from Lemma \ref{lK2} and Corollary \ref{c:approx-factor} (see proof of
\cite[Theorem 1.1]{KSV4}).

If the process $X$ is not transient, we can use argument similar as in the
proof of \cite[Theorem 1.2, p. 17]{KM} where it is shown how to deduce Harnack inequality in dimensions $d=1,2$ from  Harnack inequality in dimension
$d\geq 3$ (since in 
the latter  case the process is always transient). 
Since we will use the argument in the proof of Theorem \ref{t:bhp} 
again, here we provide the detail for the readers' convenience. 

We use the notation $\tilde{x}=(x^1,\ldots,x^{d-1})$
for $x=(x^1,\ldots,x^{d-1},x^d)\in \R^d$ and  $X=((\widetilde{X}_t,X^d_t), \P_{(\widetilde{x}, x^d)})$. As in the proof of \cite[Theorem 1.2, p. 17]{KM}, we have that for every $x^d \in \R$, 
$\widetilde{X}=(\widetilde{X}_t,\P_{\widetilde{x}})$ is a $(d-1)$-dimensional
subordinate Brownian motion with characteristic exponent
$ \widetilde{\Phi}(\tilde{\xi})=\phi(|\tilde{\xi}|^2) $ for 
$\tilde{\xi}\in \R^{d-1}.$
 
Suppose  (i) is true for 
for some $d\geq 2$ and let 
 $D$ be an open subset of $\R^{d-1}$ and   $u, v \colon \R^{d-1} \rightarrow
[0,\infty)$ be  functions that are regular harmonic in $D\cap B(
\wt {x}_0,r)$ 
 with respect to $\widetilde{X}$ and vanish 
   on $D^c \cap B(
\wt {x}_0,r)$   a.e. with respect to $(d-1)$-dimensional Lebesgue measure.

Let $f$ and $g$ $\colon \R^{d} \rightarrow
[0,\infty)$ be defined by 
$$f(\wt x, x^d)=u(\wt x)\quad \text{ and } \quad g(\wt x, x^d)=v(\wt x).$$ 
Since $$\tau_{(B(\wt {x}_0,s) \cap D)\times \R}=\inf\{t>0 : \wt
{X}_t \notin B(\wt {x}_0,s)\cap D \},$$ by the strong Markov property, $f$ and $g$ are regular harmonic in $B(
\wt {x}_0,r) \times \R$ with respect to $X$.    Clearly  $f$ and $g$ vanish 
   on $(B(
\wt {x}_0,r) \times \R) \cap (D \times \R)^c$   a.e. with respect to $d$-dimensional Lebesgue measure.
Thus, 
by applying the result to  $f$ and $g$, we see that there exists a constant $c>0$  such
that for
all $\wt {x}_0 \in \R^{d-1}$, open set $D \subset \R^{d-1}$ and $r\in (0,1)$
	\[
		\frac{u(\wt {x}_1)}{v(\wt {x}_1)}=\frac{f((\wt {x}_1, 0))}{g((\wt {x}_1, 0))}\leq c\, \frac{f((\wt {x}_2, 0))}{g((\wt {x}_2, 0))}=
c\ \frac{u(\wt {x}_2)}{v(\wt {x}_2)} \ \text{ for all }\ \wt{x}_1,\wt{x}_2\in
 D\cap B(\wt {x}_0,\tfrac{r}{2}).
	\]

Applying this argument first to $d=3$ and then to $d=2$, we finish the proof of
the theorem.
\end{proof}

\section{Green function estimates on bounded Lipschitz domain}\label{s:3}

The purpose of this section is to establish sharp two-sided Green function
estimates for $X$
in any bounded Lipschitz domain $D$ of $\R^d$. 

Recall that we have assumed that $X=(X_t,\P_x)$ is the
subordinate process
 defined by $X_t=W_{S_t}$ where $W=(W_t,\P_x)$ is  a Brownian motion in $\R^d$
independent of the  subordinator $S$ and 
the Laplace exponent $\phi$ of the  subordinator $S$
satisfies {\bf(A-1)}--{\bf(A-3)}. In this section we further assume that $X$ is transient and  that
{\bf(A-4)} also holds. 

We will first establish the interior estimates using Proposition
\ref{prop:pot-alpha0} and Theorem 
\ref{T:Har}. As in \cite{KSV2}, once we have the interior estimates, we can
apply
Theorem \ref{T:Har} and
the boundary Harnack principle (Theorem \ref{UBHP}), and use the
arguments of \cite{B3, H} to get the full estimates
for bounded Lipschitz domain $D$.

\begin{lemma}\label{GI}
For every bounded domain $D\subset \R^d$, there exists a constant
$C_2=C_2(d,\phi,$ $\text{diam}(D))>0$ such that 
  \begin{equation}\label{ub}
 G_D(x,y) \le C_3 \frac{|x-y|^{-d-2}\phi'(|x-y|^{-2})}{\phi(|x-y|^{-2})^2}\,
\qquad
 \mbox{ for all }x,y \in D\, ,
  \end{equation} and 
for all $x, y \in D$ with
$b_2^{-1}|x-y| \le
\delta_D(x) \wedge \delta_D(y)$
\begin{equation}\label{lb}
G_D(x,y) \,\ge\, C_3^{-1}\,
\frac{|x-y|^{-d-2}\phi'(|x-y|^{-2})}{\phi(|x-y|^{-2})^2}
\end{equation}
where $b_2\in (0,\frac{1}{2})$ is the constant 
from Proposition \ref{p:green}. 
\end{lemma}

\pf
 Since $G_D(x,y) \le g(|x-y|)$ and $D$ is bounded,  
 \eqref{ub} is an immediate consequence of Proposition \ref{prop:pot-alpha0}.
 
 Now we show \eqref{lb}. 
We have two cases:

Case 1: $|x-y|\leq b_2$.

Since $B(x,
b_2^{-1}|x-y|)\subset D$ and  $y\in A(x,|x-y|,b_2^{-1}|x-y|)$, we can
use Proposition \ref{p:green} to get 
\begin{align*}
 G_D(x,y)&\geq G_{B(x,b_2^{-1}|x-y|)}(x,y)\geq
c_1\tfrac{
b_2^{d+2}|x-y|^{-d-2}\phi'(b_2^{2}|x-y|^{-2})}{\phi(b_2^{2}|x-y|^
{ -2})}\E_x[\tau_{B(x,b_2^{-1}|x-y|)}]\\
&\geq c_2\tfrac{|x-y|^{-d-2}\phi'(|x-y|^{-2})}{\phi(|x-y|^{-2})^2},
\end{align*}
where in the last inequality we have used Proposition \ref{p:exittime}, 
{\bf (A-3)} and the facts that $b_2\in (0,\frac{1}{2})$ and that the function $r\mapsto
\frac{1}{\phi(r)}$ is decreasing. 

Case 2: $|x-y|> b_2$.

In this case it follows that $\delta_D(x)\wedge \delta_D(y)> 1$. 
Let $x_0\in \partial B(y,b_2)$. 
Then \[b_2^{-1}|x_0-y|=1< \delta_D(x)\wedge \delta_D(y)\]
and so, by the Case 1, we obtain
\begin{equation}
 G_D(x_0,y)\geq c_2\tfrac{b_2^{-d-2}\phi'(b_2^{-2})}{\phi(b_2^{-2})^2}.
\label{eq:green_d01}
\end{equation}

Since $G_D(\cdot ,y)$ is harmonic in $B(x_0,\frac{b_2}{2})\cup
B(x,\frac{b_2}{2})$ (with respect to $X$), we can use 
Proposition \ref{T:Har} to deduce
\begin{align}
 G_D(x,y)&=\E_x[G_D(X_{\tau_{B(x,b_2/4)}},y)]\geq
\E_x[G_D(X_{\tau_{B(x,b_2/4)}},y);X_{\tau_{B(x,b_2/4)}}\in
B(x_0,\tfrac{b_2}{4})]\nonumber\\
&\geq c_3G_D(x_0,y)\P_x(X_{\tau_{B(x,b_2/4)}}\in
B(x_0,\tfrac{b_2}{4}))\,.\label{eq:green_d02}
\end{align}
By Proposition \ref{p:Poisson1}, (\ref{eq:sub-105}) and (\ref{eq:sub-10}) we get
\begin{align}
 \P_x(X_{\tau_{B(x,b_2/4)}}\in
B(x_0,\tfrac{b_2}{4}))&=\int\limits_{B(x_0,\tfrac{b_2}{4})}K_{
B(x , \frac{b_2}{4})}(x,z)\,dz\nonumber\\
& \geq
\tfrac{c_4}{\phi(16b_2^{-2})}\int\limits_{B(x_0,\tfrac{b_2}{4})}j(|z-x|)\,dz.
\label{eq:green_d03}
\end{align}
Since 
$|z-x|\leq \text{diam}(D)$,
by the monotonicity of $j$ we deduce
\[
 \P_x(X_{\tau_{B(x,b_2/4)}}\in
B(x_0,\tfrac{b_2}{4}))\geq
c_5\tfrac{b_2^dj\left(
\text{diam}(D)\right)}{\phi(16b_2^{-2})}\,
.
\]
Therefore, using (\ref{eq:green_d01})--(\ref{eq:green_d03}) we conclude that
$$ G_D(x,y)\geq c_6 \geq c_7 \tfrac{|x-y|^{-d-2}\phi'(|x-y|^{-2})}{\phi(|x-y|^{-2})^2}.$$ 
In the last inequality we use the fact that $b_2<|x-y|\leq \text{diam}(D)$ and Corollary \ref{c:new1}. 
\qed

An open set $D$ is said to be Lipschitz
domain
if there is a localization radius
$R_1>0$  and a constant
$\Lambda >0$
such that
for every $z\in \partial D$, there is a
Lipschitz  function
$\phi_z: \R^{d-1}\to \R$ satisfying
\[| \phi_z (x)- \phi_z (w)| \leq \Lambda
|x-w| ,\] and an orthonormal coordinate
system $CS_z$ with origin at $z$ such that
$$
B(z, R_1)\cap D=B(z, R_1)\cap \{ y=
(\tilde y, y_d)\mbox{ in } CS_z: y_d > \phi_z (\tilde y) \}.
$$
The pair $(R_1, \Lambda)$ is called the
characteristics of the Lipschitz domain $D$.

Unlike \cite{KSV2} we will assume that $D$ is a bounded 
Lipschitz domain $D$ instead of $\kappa$-fat open set. 
The main reason we assume that $D$ is a bounded 
Lipschitz domain $D$ is Theorem \ref{T:Har} and the Harnack chain argument. 
Note that in \cite{KSV2}, \cite[Theorem 2.14]{KSV2} is used instead of   Theorem \ref{T:Har} and the  Harnack chain argument.  Unfortunately, it seems that, under our assumptions, the such result is not true for certain harmonic function like $u(x):=\P_{x}(X_{\tau_{B(x_1,r)}}
\in B(x_0, r) )$ when distance between $x_0$ and $x_1$ is large and $r$ is small.

\begin{lemma}\label{G:g3}
For every $L >0$ and bounded 
Lipschitz domain $D$ with the 
characteristics
 $(R_1,\Lambda)$, 
 there exists $c=c(L, d, \phi, R_1, \Lambda, \text{diam}(D))>0$ such
that for every $x,y\in D$ with  $|x-y| \le L ( \delta_D(x) \wedge \delta_D(y))
$,
\begin{equation}\label{e:g3}
G_D(x,y) \ge c \frac{|x-y|^{-d-2}\phi'(|x-y|^{-2})}{\phi(|x-y|^{-2})^2}.
\end{equation}
\end{lemma}

\pf 
By symmetry of $G_D$ we may assume $\delta_D(x) \le \delta_D(y)$.
Moreover, by Lemma \ref{GI} we can assume that
$L > b_2$ 
and 
so
we only need to show \eqref{e:g3} for $b_2
\delta_D(x) \le |x-y| \le L \delta_D(x)$.

Choose a point $w \in \partial B(x , b_2\delta_D(x))$. Then Lemma \ref{GI} gives
$$
G_D(x,w) \,\ge\, c_1
\frac{(b_2\delta_D(x))^{-d-2}\phi'((b_2\delta_D(x))^{-2})}{
\phi((b_2\delta_D(x))^{-2})^2}.
$$
Since $|y-w| \le |x-y|+|x-w| \le (L+1) \delta_D(x)$ and
$G_D(x,\,\cdot\,)=G_D(\,\cdot\,, x)$ is harmonic with respect to $X$ in
$B(y, b_2\delta_D(x)) \cup B(w, b_2\delta_D(x))$, using the assumption that
$D$ is a bounded 
 Lipschitz domain,
Theorem \ref{T:Har} and  
 the Harnack chain argument we obtain
\begin{eqnarray*}
G_D(x,y)  \ge c_2 G_D(x,w)
\ge c_3
\frac{(b_2\delta_D(x))^{-d-2}\phi'((b_2\delta_D(x))^{-2})}{
\phi((b_2\delta_D(x))^{-2})^2}.
\end{eqnarray*}
by 
Corollary \ref{c:new1}
\begin{eqnarray*}
G_D(x,y) & \ge &c_2 G_D(x,w)
\ge c_3
\frac{(b_2\delta_D(x))^{-d-2}\phi'((b_2\delta_D(x))^{-2})}{
\phi((b_2\delta_D(x))^{-2})^2} \nonumber\\
& \ge & c_4  \frac{|x-y|^{-d-2}\phi'(|x-y|^{-2})}{\phi(|x-y|^{-2})^2}.
\end{eqnarray*}
\qed

For the remainder of this section, we assume that $D$ is a bounded
Lipschitz domain with characteristics $(R_1, \Lambda)$.

Without loss of generality we may assume that $R_1\le \frac{1}{4}$.
Since $D$ is Lipschitz, there exists $\kappa=\kappa(\Lambda)\in (0, \frac{1}{2})$ such
that  for each $Q \in \partial D$ and $r \in (0, R_1)$, there exists a point 
\[A_r(Q)\in D \cap B(Q,r)\ \text{ satisfying }\ B(A_r(Q),\kappa r)
\subset D \cap B(Q,r)\,.\]

Recall that $G_D(\cdot, y)$ is
regular harmonic in $D\setminus \overline{B(y,\varepsilon)}$ for
every $\varepsilon >0$ and vanishes outside $D$.
 
Fix $z_0 \in D$ with $\kappa R_1  < \delta_D(z_0) < R_1$
and set $\eps_1:=  \frac{\kappa R_1}{24}$. Define
\[
 r(x,y): = \delta_D(x) \vee \delta_D(y)\vee |x-y| ,\ x,y\in D
\]
and
\begin{equation} \label{d:gz1}
\BB(x,y):=
\begin{cases}
\left\{ A \in D:\, \delta_D(A) > \frac{\kappa}{2}r(x,y), \,
|x-A|\vee
|y-A| < 5 r(x,y)  \right\}& \text{ if } r(x,y) <\eps_1 \\
\{z_0 \}& \text{ if } r(x,y) \ge \eps_1 .
\end{cases}
\end{equation}
Note that for every $(x,y) \in D \times D$ with $r(x,y) <\eps_1$
\begin{equation}\label{e:ar}
\tfrac16 \delta_D(A) \,\le\,r(x,y)\, \le\, 2 \kappa^{-1}  \delta_D(A), \qquad A
\in \BB(x,y).
\end{equation}
Set
$$
C_{4}:=
C_3 \text{diam}(D) 
(\tfrac{\delta_D(z_0)}{2})^{-d-3}
\frac{\phi'((\tfrac{\delta_D(z_0)}{2})^{-2})}{\phi((\tfrac{\delta_D(z_0)}{2})^{
-2 } )^2 }
\, .
$$
By \eqref{ub} and 
Corollary \ref{c:new1} (with $a=2\text{diam}(D) /\delta_D(z_0)$ and $b=1$) we see that 
\[G_D(x, z_0)\leq C_4 \ \text { for }\ \ x\in D \setminus B(z_0,
\tfrac{\delta_D(z_0)}{2}).
\]
Now we define
\begin{equation}\label{d:gz0} g_D(x ):=  G_D(x, z_0) \wedge
C_4.
\end{equation}
We note that for $\delta_D(z) \le 6 \eps_1$, 
\[
 g_D(z )= G_D(z, z_0), 
\]
since $6\eps_1<\frac{\delta_D(z_0)}{4}$ and thus 
$|z-z_0| \ge \delta_D(z_0) - 6 \eps_1 \ge \frac{\delta_D(z_0)}{2}$.

The following lemma follows
 from Theorem \ref{T:Har} and the standard
Harnack
chain argument:

\begin{lemma}\label{G:g2}
There exists $c>1$ such that for every $x \in D$ satisfying
$\delta_D(x)\ge\frac{ \kappa^3\eps_1}{64}$ we have 
\[c^{-1}  \le g_D(x) \le c\,.\]
\end{lemma}
\medskip

\begin{thm}\label{t:Gest}
Suppose $X$ is transient and $\phi$
satisfies {\bf(A-1)}--{\bf(A-4)}. If $D$ is a
bounded Lipschitz domain with characteristics $(R_1, \Lambda)$,
then there exists
$c=c($diam$(D), R_1, \Lambda, \phi)>1$
such that for every $x, y \in D$ and $A \in \BB(x,y)$
\begin{equation}\label{e:Gest}
c^{-1}\,\tfrac{g_D(x) g_D(y)\phi'(|x-y|^{-2})}{g_D(A)^2|x-y|^{d+2}\phi(|x-y|^{-2})^2}\,\le\,
G_D(x,y) \,\le\, c\,\tfrac{g_D(x) g_D(y)\phi'(|x-y|^{-2})}{g_D(A)^2|x-y|^{d+2}\phi(|x-y|^{-2})^2},
\end{equation}
where 
$g_D$ and $\BB(x,y)$ are defined by \eqref{d:gz0} and
\eqref{d:gz1} respectively.
\end{thm}

\pf
Since the proof is an adaptation of the proofs of
\cite[Proposition 6]{B3} and \cite[Theorem 2.4]{H}, we only give the
proof when $\delta_D(x) \le \delta_D(y)\le \frac{\kappa}{4} |x-y|$. In this
case, we have $r(x,y)=|x-y|$

By Theorem \ref{T:Har},  we see that for all $x, y \in D$ and $A_1, A_2 \in
\BB(x,y)$,
\[g_D(A_1)\ \text{ is comparable to }\  g_D(A_2)\,.
\]

Set $r= \frac{|x-y| \wedge
\eps_1}{2}$ and choose 
\[
 Q_x, Q_y \in \partial D\ \ \text{ with }\ \ |Q_x-x|
=\delta_D(x)\ \ \text{ and }\ \ |Q_y-y| =\delta_D(y)\,.
\]
Pick points $x_1=A_{\kappa
r/2}(Q_x)$ and $y_1=A_{\kappa r/2}(Q_y)$ so that \[x, x_1 \in B(Q_x,
\kappa r/2)\ \ \text{ and }\ \ y, y_1 \in B(Q_y, \kappa r/2)\,.\] Then one can
easily check that 
$|z_0-Q_x| \ge \kappa r$  and $|y-Q_x| \ge r$.

Then Theorem \ref{UBHP} implies
$$
c_1^{-1}\,  \frac{G_D(x_1,y)}{g_D(x_1)}\,\le\, \frac{G_D(x,y)}{g_D(x)}
\,\le\, c_1 \frac{G_D(x_1,y)}{g_D(x_1)}
$$
for some $c_1>1$. 
 
Also, since  $|z_0-Q_y| \ge r$ and
$|x_1-Q_y| \ge r$, by  Theorem \ref{UBHP} again,
$$
c_1^{-1}\,\frac{G_D(x_1,y_1)}{g_D(y_1)} \,\le\,
\frac{G_D(x_1,y)}{g_D(y)} \,\le\, c_1 \frac{G_D(x_1,y_1)}{g_D(y_1)}.
$$
Therefore
$$
c_1^{-2}\,\frac{G_D(x_1,y_1)}{g_D(x_1)g_D(y_1)}  \,\le\,
\frac{G_D(x,y)}{g_D(x)g_D(y)} \,\le\, c_1^2\frac{G_D(x_1,y_1)}
{g_D(x_1)g_D(y_1)}.
$$

Now we can use Lemma \ref{G:g3} for the lower and Lemma \ref{GI} for the upper bound  to get
\begin{align}
\tfrac{c_2^{-1} c_1^{-2}}{g_D(x_1)g_D(y_1)}
\tfrac{|x_1-y_1|^{-d-2}\phi'(|x_1-y_1|^{-2})}{\phi(|x_1-y_1|^{-2})^2}
  \,\le\,
\frac{G_D(x,y)}{g_D(x)g_D(y)}
 &\le\, \tfrac{c_2
c_1^2}{g_D(x_1)g_D(y_1)}\tfrac{|x_1-y_1|^{-d-2}\phi'(|x_1-y_1|^{-2})}{
\phi(|x_1-y_1|^{-2})^2}\label{e:GE0}
\end{align}
for some $c_2>1$.

Since $\frac{|x-y|}{3} < |x_1-y_1| < 2 |x-y| $,  
Corollary \ref{c:new1}  yields
\begin{align*}
\tfrac{|x_1-y_1|^{-d-2}\phi'(|x_1-y_1|^{-2})}{\phi(|x_1-y_1|^{-2})^2}
\le     
2 \cdot  
3^{d+3} \tfrac{|x-y|^{-d-2}\phi'(9|x-y|^{-2})}{\phi(9|x-y|^{-2})^2}
\le 
2 \cdot  
3^{d+3} \tfrac{|x-y|^{-d-2}\phi'(|x-y|^{-2})}{\phi(|x-y|^{-2})^2}
\end{align*}
and
\begin{align*}
\tfrac{|x_1-y_1|^{-d-2}\phi'(|x_1-y_1|^{-2})}{\phi(|x_1-y_1|^{-2})^2}
\ge   \ 
3^{-1} \cdot 
2^{-d-3}
\tfrac{|x-y|^{-d-2}\phi'(4^{-1}|x-y|^{-2})}{\phi(4^{-1}|x-y|^{-2})^2}
\ge  
3^{-1} \cdot 
2^{-d-3} \tfrac{|x-y|^{-d-2}\phi'(|x-y|^{-2})}{\phi(|x-y|^{-2})^2}.
\end{align*}
Therefore, 
\begin{align}
\tfrac{2^{-d-3}c_2^{-1} c_1^{-2}}{3g_D(x_1)g_D(y_1)}
\tfrac{|x-y|^{-d-2}\phi'(|x-y|^{-2})}{\phi(|x-y|^{-2})^2}
  \,\le\,
\tfrac{G_D(x,y)}{g_D(x)g_D(y)}
\,\le\, \tfrac{2 \cdot 3^{d+3}c_2
c_1^2}{g_D(x_1)g_D(y_1)}\tfrac{|x-y|^{-d-2}\phi'(|x-y|^{-2})}{\phi(|x-y|^{-2})^2
}
.\label{e:GE1}\end{align}

If $r=\frac{\eps_1}{2}$, then $r(x,y)=|x-y| \ge \eps_1$ and so
\[
g_D(A)=g_D(z_0)=C_4\ \ \text{ and }\ \ \delta_D(x_1)\wedge \delta_D(y_1) \ge
\tfrac{\kappa^2 r}{2} = \tfrac{\kappa^2 \eps_1}{4}\,.\]

Thus, in this case,  Lemma \ref{G:g2} yields
\begin{equation}\label{e:GE2}
 c_3^{-1} \le\frac{g_D(A)^2}
{g_D(x_1)g_D(y_1)} \le c_3
\end{equation}
for some $c_3>1$.

In the case  $r<\frac{\eps_1}{2}$ we have  $r(x,y)=|x-y| < \eps_1$ and $r=\frac12
r(x,y)$. Hence \[\delta_D(x_1)\wedge \delta_D(y_1) \ge 
\tfrac{\kappa^2 r}{2} = \tfrac{\kappa^2 r(x,y)}{4}\,.\]

Since $|x_1-A|\vee |y_1-A| \le 5r(x,y)+|x_1-x|+|y_1-y|\leq 5r(x,y)+2\kappa
r\le 6r(x,y)$,
 Theorem \ref{T:Har} applied to $g_D$ 
gives
\begin{equation}\label{e:GE3}
c^{-1}_4 \,\le\,\frac{g_D(A)}{g_D(x_1)} \,\le\,c_4 \quad \mbox{and}
\quad c^{-1}_4 \,\le\,\frac{g_D(A)}{g_D(y_1)} \,\le\,c_4
\end{equation}
for some constant $c_4>0$. Combining
\eqref{e:GE1}-\eqref{e:GE3}, we get
\begin{align*}
&c_5^{-1}\,\tfrac{g_D(x)
g_D(y)}{g_D(A)^2}\tfrac{|x-y|^{-d-2}\phi'(|x-y|^{-2})}{\phi(|x-y|^{-2})^2}\le
G_D(x,y) 
\le c_5\,\tfrac{g_D(x)
g_D(y)}{g_D(A)^2}\tfrac{|x-y|^{-d-2}\phi'(|x-y|^{-2})}{\phi(|x-y|^{-2})^2}
\end{align*}
for all   $A \in \BB(x,y)$.
 \qed

\section{Explicit Green function estimates on bounded $C^{1,1}$-open
sets}\label{s:4}

The purpose of this section is to 
establish 
the  explicit Green function estimates from Theorem \ref{t:Gest} in
the case of  bounded $C^{1,1}$ open sets.

\begin{thm}\label{t:Ge} 
Suppose that $X=(X_t:\, t\ge 0)$ is a transient $d$-dimensional
subordinate
Brownian motion 
where the corresponding subordinator $S$ has the Laplace exponent $\phi$
satisfying {\bf(A-1)}--{\bf(A-5)}.
If 
$D$ is a bounded $C^{1,1}$ domain in $\R^d$ with $C^{1,1}$ characteristics
$(R, \Lambda)$, 
then there exists
$
c=
c(R, \Lambda, \phi,\textrm{diam}(D))>0$ such that
\begin{equation}\label{e:z1}
c^{-1} \,\left(V(\delta_D(x))
   \wedge 1 \right)\,\le \,g_D(x)\, \le\,
c\,\left( V(\delta_D(x))   \wedge 1 \right)  \ \ \text{ for all }\ \ \ x
\in D.
\end{equation}
\end{thm}
\pf
    The proof follows  the proof of \cite[Theorem 4.6]{KSV2} by using
our Proposition \ref{prop:pot-alpha0}, Lemma \ref{L:2} and  Theorem \ref{UBHP}\,.
\qed

\medskip

\noindent {\bf Proof of  Theorem \ref{t:Gest2}}.
Only using the fact that  $V$ is increasing  and subadditive, the following is proved in \cite[(4.38)]{KSV2}.
\begin{equation}\label{new2}
\frac{(V(\delta_D(x)) \wedge 1)(V(\delta_D(y))\wedge 1)}
{(V(\delta_D(x)\vee\delta_D(y)\vee|x-y|) \wedge 1)^2}
\asymp \frac{V(\delta_D(x))V(\delta_D(y))}
{V^2(\delta_D(x)\vee\delta_D(y)\vee|x-y|)}.
\end{equation}
Thus, when $D$ is connected, Theorem \ref{t:Gest2} follows from
  \eqref{new2} and our Theorems
\ref{t:Gest} and \ref{t:Ge}.

 Next we assume that $D$ is not connected. The proof below is similar to the one
 in \cite[Theorem 3.4]{CKSV2}. 
 
 Let $(R, \Lambda)$ be the
$C^{1,1}$ characteristics of $D$. Note that $D$ has only finitely
many components and the distance between any two distinct components
of $D$ is at least $R>0$. 

Assume first that $x$ and $ y$ are  in
two  distinct  components of $D$. Let $D(x)$ be the component of $D$
that contains $x$. Then by the strong Markov property and \eqref{eq:sub-105} we
obtain
$$
G_D(x, y)= \E_x \left[ G_D\big(X_{\tau_{  D(x)}}, y\big) \right] =
\E_x \left[ \int\limits_0^{\tau_{ D(x)}} \left( \int\limits_{D\setminus D(x)}
j (|X_s-z|) G_D(z, y) dz \right)  ds \right].
$$
Consequently,
\begin{align}\label{e:3.12}
j({\rm diam} (D)) \, \E_x[\tau_{D(x)}] \, \int\limits_{D\setminus D(x)}
G_D(y, z) dz
 & \leq G_D(x, y) \leq j(R) \, \E_x[\tau_{
D(x)}] \, \int\limits_{D\setminus D(x)} G_D(y, z) dz.
\end{align}
Applying the two-sided estimates \eqref{e:Gest21-alt1} established in the
first part of this proof to $D(x)$, after integrating out the second variable we get
\begin{align}\label{e:3.13}
\tfrac{c_1^{-1}}{\sqrt{\phi(\delta_D(x)^{-2})}} =\tfrac{c_1^{-1}}{\sqrt{
\phi(\delta_{D(x)}(x)^{-2})}} \le \E_x
\left[ \tau_{D(x)} \right] \le \tfrac{c_1^{-1}}{\sqrt{
\phi(\delta_{D(x)}(x)^{-2})}}=\tfrac{c_1^{-1}}{\sqrt{\phi(\delta_D(x)^{-2})}}
.
\end{align}
By \eqref{e:3.13} we get
$$
\int\limits_{D\setminus D(x)} G_D(y, z) dz \geq \int\limits_{D(y)} G_{D(y)}
(y, z) dz =\E_y [\tau_{D(y)}] \geq 
\tfrac{c_2}{\sqrt{\phi(\delta_D(y)^{-2})}}.
$$
On the other hand, \eqref{eq:sub-105} and \eqref{e:3.13}
 imply
\begin{eqnarray*}
\int\limits_{D\setminus D(x)} G_D(y, z) dz 
 &\leq & \E_y \big[
\tau_D\big]  = \E_y \Big[ \tau_{D(y)}\Big]
+ \E_y \Big[\E_{X_{\tau_{D(y)}}} [\tau_D ]\Big]\\
&\leq & \tfrac{c_3}{\sqrt{\phi(\delta_D(y)^{-2})}} +
\E_y \left[ \int\limits_0^{\tau_{D(y)}}
\int\limits_{D\setminus D(y)}
j (|X_s- z|) \, \E_z [\tau_D] dz ds \right] \\
&\leq & \tfrac{c_3}{\sqrt{\phi(\delta_D(y)^{-2})}} 
+  j(R)\E_y \big[ \tau_{D(y)}\big]
|D| 
 \, \E_0 [\tau_{B(0, \text{diam(D)})}]\\
 &\leq & \tfrac{c_3}{\sqrt{\phi(\delta_D(y)^{-2})}} 
+ c_4
  \E_y \big[ \tau_{D(y)}\big]
\leq \tfrac{c_5}{\sqrt{\phi(\delta_D(y)^{-2})}}\,.
\end{eqnarray*}
We conclude from the last three displays and  \eqref{e:3.12} that there is a
constant
$c_6\geq 1$ such that
\begin{equation}\label{e:dfc}
\tfrac{c_6^{-1}}{\sqrt{\phi(\delta_D(x)^{-2})
\phi(\delta_D(y)^{-2})}} \leq
G_D(x, y) \leq
 \tfrac{c_6}{\sqrt{\phi(\delta_D(x)^{-2})
\phi(\delta_D(y)^{-2})}}.
\end{equation}
Noting that \[R\leq
 |x-y| \leq {\rm diam}(D)\] when $x$ and $y$ are in different components 
 of 
$D$, by Corollary \ref{c:new1} we obtain \eqref{e:Gest21-alt1}.

Now we assume that $x, y$ are in the same component $U$ of $D$.
Applying \eqref{e:Gest21-alt1} to $U$ we get
\begin{align*}
G_D(x,y) \ge G_U(x,y) &\ge  c_7 \left(1 \wedge
\tfrac{\phi(|x-y|^{-2})}{\sqrt{\phi(\delta_U(x)^{-2})
\phi(\delta_U(y)^{-2})}}\right)\,
\tfrac{|x-y|^{-d-2}\phi'(|x-y|^{-2})}{\phi(|x-y|^{-2})^2}
\\
&=
c_7 \left(1 \wedge
\tfrac{\phi(|x-y|^{-2})}{\sqrt{\phi(\delta_D(x)^{-2})
\phi(\delta_D(y)^{-2})}}\right)\,
\tfrac{|x-y|^{-d-2}\phi'(|x-y|^{-2})}{\phi(|x-y|^{-2})^2}
.
\end{align*}
For the upper bound, we use the strong Markov property, \eqref{eq:sub-105} and
\eqref{e:3.13}--\eqref{e:dfc} to get
\begin{eqnarray}
&&G_D(x,y)\nn\\&=&\, G_U(x,y) +\E_x \left[ G_D(
X_{\tau_{  U}}, y) \right] \nonumber\\
&\le& c_8 \left(1 \wedge
\tfrac{\phi(|x-y|^{-2})}{\sqrt{\phi(\delta_D(x)^{-2})
\phi(\delta_D(y)^{-2})}}\right)\,
\tfrac{|x-y|^{-d-2}\phi'(|x-y|^{-2})}{\phi(|x-y|^{-2})^2}+
\E_x \left[ \int\limits_0^{\tau_{ U}}\int\limits_{D\setminus U}
j (|X_s-z|) G_D(z, y) dzds \right] \nonumber\\
&\le& c_8  \left(1 \wedge
\tfrac{\phi(|x-y|^{-2})}{\sqrt{\phi(\delta_D(x)^{-2})
\phi(\delta_D(y)^{-2})}}\right)\,
\tfrac{|x-y|^{-d-2}\phi'(|x-y|^{-2})}{\phi(|x-y|^{-2})^2}+
j(R) \, \E_x[\tau_{
U}] \, \int\limits_{D\setminus U} G_D(y, z) dz \nonumber\\
&\le& c_8  \left(1 \wedge
\tfrac{\phi(|x-y|^{-2})}{\sqrt{\phi(\delta_D(x)^{-2})
\phi(\delta_D(y)^{-2})}}\right)\,
\tfrac{|x-y|^{-d-2}\phi'(|x-y|^{-2})}{\phi(|x-y|^{-2})^2}
+  \tfrac{c_{9} \int\limits_{D\setminus U}
\tfrac{dz}{\sqrt{\phi(\delta_D(z)^{-2})}}dz}{\sqrt{\phi(\delta_D(x)^{-2})
\phi(\delta_D(y)^{-2})}}\,
.
\label{e:fip}
\end{eqnarray}
Since $D$ is bounded, we get
\begin{eqnarray}
\tfrac{1}{\sqrt{\phi(\delta_D(x)^{-2})\phi(\delta_D(y)^{-2})}}
\int\limits_{D\setminus U} \tfrac{dz}{\sqrt{\phi(\delta_D(z)^{-2})}}
 &\le& \tfrac{|D|}{\sqrt{\phi(\delta_D(x)^{-2})
\phi(\delta_D(y)^{-2})
\phi(\text{diam}(D)^{-2})}}\nn\\
 &\le&c_{10} \left(1 \wedge
\tfrac{\phi(|x-y|^{-2})}{\sqrt{\phi(\delta_D(x)^{-2})
\phi(\delta_D(y)^{-2})}}\right), \nn
\end{eqnarray}
which together with  \eqref{e:fip} and Corollary \ref{c:new1} gives
\begin{eqnarray*}
G_D(x,y)& \le&
c_{11}\left(1 \wedge
\tfrac{\phi(|x-y|^{-2})}{\sqrt{\phi(\delta_D(x)^{-2})
\phi(\delta_D(y)^{-2})}}\right) \\
& \le&
c_{12} \left(1 \wedge
\tfrac{\phi(|x-y|^{-2})}{\sqrt{\phi(\delta_D(x)^{-2})
\phi(\delta_D(y)^{-2})}}\right)\,
\tfrac{\phi'(|x-y|^{-2})}{|x-y|^{d+2}\phi(|x-y|^{-2})^2} .
\end{eqnarray*}
\qed

\noindent {\bf Proof of  Theorem \ref{t:bhp}}.
When $d=1$, the theorem follows from Proposition \ref{e:behofV}, Theorem \ref{t:Sil} and  Theorem \ref{UBHP} (i). 

Note that the result in \cite[Lemma 4.2]{CKSV} is true in our case too. 
By this result,
Theorem \ref{T:Har}, Theorem \ref{UBHP} (i) and Theorem \ref{t:Gest2}, the proof of  Theorem
\ref{t:bhp} is the same as the 
proof of 
\cite[Theorem 1.3]{KSV2}  when $d\geq
3$.

Note that if $D$ is a $C^{1,1}$ open set 
in $\R^{d-1}$ with characteristics $(R, \Lambda)$,  then $D \times \R$ is clearly $C^{1,1}$ open set 
in $\R^{d}$ with the same characteristics $(R, \Lambda)$. Thus 
the case $d= 2$ can be handled in the same way as in the proof of Theorem
\ref{UBHP} (i)\,.
\qed

\vspace{.3cm} \noindent {\bf Acknowledgment}: 
We thank the referee for many valuable comments and suggestions.

\vspace{.1in}

\providecommand{\bysame}{\leavevmode\hbox to3em{\hrulefill}\thinspace}
\providecommand{\MR}{\relax\ifhmode\unskip\space\fi MR }
\providecommand{\MRhref}[2]{%
  \href{http://www.ams.org/mathscinet-getitem?mr=#1}{#2}
}
\providecommand{\href}[2]{#2}

%\bibliography{myrefs}

\end{document}